\documentclass[10pt]{amsart}
\usepackage{amsmath, amssymb}
\usepackage{amsfonts}
\usepackage{mathrsfs}
\usepackage[arrow,matrix,curve,cmtip,ps]{xy}
\usepackage{tikz-cd}

\usepackage{epsfig}

\usepackage{amsthm}

\usepackage{scalerel,stackengine} 
\stackMath
\newcommand\reallywidehat[1]{%
\savestack{\tmpbox}{\stretchto{%
  \scaleto{%
    \scalerel*[\widthof{\ensuremath{#1}}]{\kern-.6pt\bigwedge\kern-.6pt}%
    {\rule[-\textheight/2]{1ex}{\textheight}}
  }{\textheight}%
}{0.5ex}}%
\stackon[1pt]{#1}{\tmpbox}%
}
\allowdisplaybreaks

\newtheorem{theorem}{Theorem}[section]
\newtheorem{lemma}[theorem]{Lemma}
\newtheorem{proposition}[theorem]{Proposition}
\newtheorem{corollary}[theorem]{Corollary}
\newtheorem{notation}[theorem]{Notation}
\newtheorem*{theorem*}{Theorem}
\theoremstyle{remark}
\newtheorem{remark}[theorem]{Remark}
\newtheorem{definition}[theorem]{Definition}
\newtheorem{example}[theorem]{Example}

\numberwithin{equation}{section}

\newcommand{\N}{\mathbb{N}}

\newcommand{\C}{\mathbb{C}}
\newcommand{\T}{\mathbb{T}}
\newcommand{\K}{\mathbb{K}}
\newcommand{\B}{\mathbb{B}}
\newcommand{\Hi}{\mathcal{H}}

\newcommand{\End}{\operatorname{End}}

\newcommand{\gae}{\lower 2pt \hbox{$\, \buildrel {\scriptstyle >}\over {\scriptstyle
\sim}\,$}}

\newcommand{\lae}{\lower 2pt \hbox{$\, \buildrel {\scriptstyle <}\over {\scriptstyle
\sim}\,$}}

\newcommand{\MU}[1]{
\setbox0\hbox{$#1$}
\setbox1\hbox{$W$}
\ifdim\wd0>\wd1 #1^{\sim} \else \widetilde{#1} \fi
}

\begin{document}
\title{Naimark's Problem for graph C*-algebras and Leavitt path algebras}

\author{Kulumani M. Rangaswamy}

\author{Mark Tomforde}

\address{Department of Mathematics \\ University of Colorado \\ Colorado Springs, CO 80918-3733 \\USA}
\email{kmranga@gmail.com}
\email{mark.tomforde@gmail.com}

\thanks{This work was supported by a grant from the Simons Foundation (\#527708 to Mark Tomforde)}


\date{\today}



\maketitle

\vspace{-0.6cm}
\begin{center}
\emph{Dedicated to Professor Laszlo Fuchs on the occasion of his 100th birthday.}
\end{center}

\begin{abstract}
We describe how boundary paths in a graph can be used to construct irreducible representations of the associated graph $C^*$-algebra and the associated Leavitt path algebra.  We use this construction to establish two sets of results:  First, we prove that Naimark's Problem has an affirmative answer for graph $C^*$-algebras, we prove that the algebraic analogue of Naimark's Problem has an affirmative answer for Leavitt path algebras, and we give necessary and sufficient conditions on the graphs for the hypotheses of Naimark's Problem to be satisfied.   Second, we characterize when a graph $C^*$-algebra has a countable (i.e., finite or countably infinite) spectrum, and prove that in this case the unitary equivalence classes of irreducible representations are in one-to-one correspondence with the shift-tail equivalence classes of the boundary paths of the graph.
\end{abstract}

\section{Introduction}

Since the fundamental Gelfand-Naimark theorem of $C^*$-algebra theory was first established in 1941, the study of representations has been of paramount importance in understanding $C^*$-algebras.  One of the earliest questions about the representation theory of $C^*$-algebras was asked by Naimark in 1951, and it is now referred to as \emph{Naimark's Problem}.

\smallskip

\noindent \textbf{Naimark's Problem:}  \emph{If a $C^*$-algebra has a unique irreducible $*$-representation, up to unitary equivalence, is that $C^*$-algebra isomorphic to the $*$-algebra of compact operators on some (not necessarily separable) Hilbert space?}

\smallskip

In the immediate years after Naimark stated his problem, progress was made for special classes of $C^*$-algebras.  An affirmative answer to Naimark's problem was obtained for GCR $C^*$-algebras (today more commonly known as Type~I $C^*$-algebras) \cite[Theorem~7.3]{Kap}, and an affirmative answer was also obtained for separable $C^*$-algebras \cite[Theorem~4]{Ros}.  However, the general question posed by the problem remained open for decades, until a major accomplishment in 2004 when Akemann and Weaver used Jensen's $\diamondsuit$~axiom, a combinatorial principle known to be independent of ZFC, to construct a counterexample to Naimark's problem that is generated by $\aleph_1$ elements~\cite{AW}.  This shows that an affirmative answer to Naimark's problem cannot be obtained using ZFC alone.  Whether Naimark's problem itself is independent of ZFC remains unknown.

In light of this counterexample, it is now important to identify classes of (nonseparable) $C^*$-algebras for which Naimark's Problem has an affirmative answer.  This not only allows us to know that these classes are well-behaved in (and to some extent classifiable by) their representation theory, but it also tells us where not to look for a (hypothetical) counterexample to Naimark's problem within ZFC.

In addition, one can ask an algebraic analogue of Naimark's question.

\smallskip

\noindent \textbf{Algebraic Naimark's Problem:}  \emph{If an associative algebra over a field $k$ has a unique irreducible algebraic representation, up to algebraic equivalence, is that algebra isomorphic to the algebra $M_\Lambda (k)$ of matrices with finitely many nonzero entries for some index set $\Lambda$?}

\smallskip

One would not expect this question to have an affirmative answer in general, but as with Naimark's Problem, it is important to identify examples and classes of associative $k$-algebras for which the answer is known.

In this paper, we show that Naimark's Problem has an affirmative answer for the class of graph $C^*$-algebras, and that the Algebraic Naimark's Problem has an affirmative answer for the class of Leavitt patgh algebras.  It is important to note that there are no restrictions on the graphs in our results, so that uncountable graphs (and hence nonseparable graph $C^*$-algebras and uncountably-infinite-dimensional Leavitt path algebras) are allowed.

In prior work \cite{ST}, the second author and Suri proved that Naimark's Problem has an affirmative answer for AF graph $C^*$-algebras and that Naimark's Problem has an affirmative answer for $C^*$-algebras of graphs in which each vertex emits countably many edges.  In this paper, we are finally able to generalize this result to obtain an affirmative answer to Naimark's problem for all graph $C^*$-algebras, as well as obtain an affirmative answer to the Algebraic Naimark's Problem for all Leavitt path algebras.  Moreover, the techniques we develop (see Section~\ref{Irreducible-repns-sec}) are completely different from the techniques used in \cite{ST}, and these new techniques provide novel insights into the representation theory of graph $C^*$-algebras and Leavitt path algebras.  In addition to our application  to Naimark's problem, we are also able to use our results on representations to obtain a characterization (see Theorem~\ref{countable-infinite-irrep-st-equiv-comp-thm} and Theorem~\ref{trichotomy-spectrum-thm}) of when a graph $C^*$-algebra has countable spectrum, and prove that  in this situation each unitary equivalence class of an irreducible representation of the graph $C^*$-algebra corresponds to a shift-tail equivalence class of a boundary path in the graph.

The paper is organized as follows:  In Section~\ref{prelim-sec} we establish definitions and notation and deduce some preliminary results.  In particular, Lemma~\ref{ideal-of-line-point-matrix-and-compacts-lem} provides sufficient conditions for an ideal in a graph $C^*$-algebra (respectively, an ideal in a Leavitt path algebra) to be isomorphic to the $*$-algebra of compact operators on some Hilbert space (respectively, the algebra of matrices with only finitely many nonzero entries for some index set).

In Section~\ref{Irreducible-repns-sec} we construct representations of Leavitt path algebras and graph $C^*$-algebra on vector spaces and Hilbert spaces determined by boundary paths of the graph.  Our presentation builds off the work of Carlsen and Sims \cite{CS} (for graph $C^*$-algebras) and Chen \cite{Che} (for Leavitt path algebras), and we stand firmly on the shoulders of these giants as we develop our results.  For our work, we need to modify and extend the results of Carlsen and Sims \cite{CS} and Chen \cite{Che} to include additional boundary paths (in particular, finite paths ending at singular vertices) in order to capture the analysis of additional representations.  Our definition of boundary paths can be found in Definition~\ref{boundary-paths-def}, and our main result on representations of graph algebras is Theorem~\ref{decomposition-thm}.

In Section~\ref{lemmas-relating-graphs-repns-sec} we establish a number of lemmas showing how the structure of a graph is related to the number of representations of either the graph $C^*$-algebra or the Leavitt path algebra.  We use these lemmas repeatedly throughout the subsequent sections.

In Section~\ref{Naimark-Problem-sec} we establish our main result regarding Naimark's Problem in Theorem~\ref{Naimark-thm}.  In particular, we show Naimark's Problem has an affirmative answer for graph $C^*$-algebras and that the Algebraic Naimark's Problem has an affirmative answer for Leavitt path algebras.  In addition, we establish that a graph $C^*$-algebra $C^*(E)$ has a unique irreducible $*$-representation, up to unitary equivalence, if and only if for any field $k$ the Leavitt path algebra $L_k(E)$ has a unique irreducible representation, up to algebraic equivalence, and we describe necessary and sufficient conditions on the graph $E$ for this uniqueness of irreducible representations to occur.  We end this section with Example~\ref{Cozzens-Kolchin-ex}, where we provide a class of counterexamples to the Algebraic Naimark's Problem, showing that there are $k$-algebras with a unique irreducible representation, up to algebraic equivalence, that are not isomorphic to $M_\Lambda(k)$ for any indexing set $\Lambda$.

In Section~\ref{countably-many-repns-sec} we generalize the hypothesis of Naimark's Problem and examine graph $C^*$-algebras that have a countable (i.e., finite or countably infinite) number of unitary equivalence classes of irreducible $*$-representations.  In Theorem~\ref{countable-infinite-irrep-st-equiv-comp-thm} we prove that the collection of unitary equivalence classes of irreducible representations of $C^*(E)$ is countable if and only if $E$ has no cycles and a countable number of shift-tail equivalence classes of boundary paths if and only if $C^*(E)$ has an elementary composition series of countable length.  Moreover, in the situation when these three equivalent conditions hold, every irreducible representation is unitarily equivalent to the representation coming from a boundary path as constructed in Theorem~\ref{decomposition-thm}.  This result implies a trichotomy for the spectrum of a graph $C^*$-algebra, which we state in Theorem~\ref{trichotomy-spectrum-thm}, describing possibilities for the cardinality of the spectrum of $C^*(E)$ and its relationship to the shift-tail equivalence classes of boundary paths, based on whether or not $E$ contains a cycle.

\section{Preliminaries} \label{prelim-sec}

For a Hilbert space $\Hi$, we let $\K (\Hi)$ denote the $C^*$-algebra of compact operators on $\Hi$, and we let $\B (\Hi)$ denote the $C^*$-algebra of bounded operators on $\Hi$.  We shall also use the convention that the word \emph{countable} shall mean either countably infinite or finite.

\begin{definition}
A \emph{graph} $E=(E^{0},E^{1},r,s)$ consists of a set of vertices $E^0$, a set of edges $E^1$, and functions $r: E^1 \to E^0$ and $s : E^1 \to E^0$ identifying the range and source of each edge.

It is important to note that we place no restrictions on the graphs in this paper.  In particular, the vertex and edge sets of our graphs are not required to be countable, and for a given graph $E$ the graph $C^*$-algebra $C^*(E)$ need not be separable.

A vertex $v \in E^0$ is called a \emph{sink} if $s^{-1}(v) = \emptyset$ and an \emph{infinite emitter} if $| s^{-1}(v) | = \infty$.  A \emph{singular vertex} is a vertex that is either a sink or an infinite emitter, and we let $E^0_\textnormal{sing}$ denote the set of singular vertices of $E$.  A vertex $v\in E^0$ is called a \emph{regular vertex} if $v$ is not a singular vertex; that is, if $0 < | s^{-1}(v) | < \infty$.  The set of regular vertices is denoted $E^0_\textnormal{reg}$, and we observe that $E^0$ is the disjoint union of $E^0_\textnormal{reg}$ and $E^0_\textnormal{sing}$.
\end{definition}

\begin{definition}
Let $E=(E^{0},E^{1},r,s)$ be a graph. The \emph{graph $C^*$-algebra} $C^*(E)$ is the universal $C^*$-algebra generated by mutually orthogonal projections $\{p_v :v\in E^0 \}$ and partial isometries with mutually
orthogonal ranges $\{s_e :e\in E^1\}$ satisfying
\begin{itemize}
\item[(1)] $s_e^* s_e = P_{r(e)}$  for all $e \in E^1$,
\item[(2)] $s_es_e^* \leq P_{s(e)}$ for all $e \in E^1$, and
\item[(3)] $p_v = \displaystyle \sum_{\{ e \in E^1 : s(e) = v \} } s_es_e^*$ whenever $v \in E^0$ with $0 < | s^{-1}(v) | < \infty$.
\end{itemize}
\end{definition}

\begin{definition}
Let $E=(E^{0},E^{1},r,s)$ be a graph and let $k$ be a field.  The \emph{Leavitt path algebra
} $L_k(E)$ is the universal $k$-algebra generated by a set $\{p_v: v \in E^0 \}$ of pairwise orthogonal idempotents together with a set of elements $\{s_e, s_e^*: e\in E^1 \}$ satisfying the following relations:
\begin{itemize}
\item[(1)] $p_{s(e)}s_e=s_e=s_e s_{r(e)}$ for all $e\in E^1$,
\item[(2)] $p_{r(e)}s_e^*=s_e^*=s_e^*P_{s(e)}$ for all $e\in E^1$,
\item[(3)] $s_e^*s_e=p_{r(e)}$ for all $e \in E^1$ and $s_e^* s_f=0$ for $e,f \in E^1$ with $e\neq f$, and
\item[(4)] $p_v = \displaystyle  \sum_{ \{ e \in  E^1 : s(e) = v \} } s_es_e^*$ whenever $v \in E^0$ with $0 < | s^{-1}(v) | < \infty$.
\end{itemize}
\end{definition}

\begin{remark}
It is common for algebraists working on Leavitt path algebras to simply use $e$, $e^*$, and $v$ in place of what we denote as $s_e$, $s_e^*$, and $p_v$.  However, in this paper we will find it convenient to use the $s_e$, $s_e^*$, and $p_v$ notation to draw parallels between the Leavitt path algebra and graph $C^*$-algebra results.
\end{remark}

Let $E$ be a graph.  A \emph{path} in $E$ (also called a \emph{finite path} in $E$) is defined to be either a vertex $v$ or a finite sequence of edges $\alpha := e_1 \ldots e_n$ with $r(e_i) = s(e_{i+1})$ for $1 \leq i \leq n-1$.  We let $E^*$ denote the set of all paths in $E$ and we call $E^*$ the \emph{finite path space of $E$}.

We extend the range map $r : E^1 \to E^0$ and the source map $s : E^1 \to E^0$ to $E^*$ as follows:  Let $\alpha \in E^*$.  If $\alpha = v$ is a vertex, we define $s(\alpha) := v$ and $r(\alpha) := v$.  If $\alpha = e_1 \ldots e_n$ is a finite sequence of edges, we define $s(\alpha) := s(e_1)$ and $r(\alpha) = r(e_n)$.  We also define the \emph{length} of a vertex to be $0$ and the length of a path $\alpha := e_1 \ldots e_n$ to be $n$.  We write $| \alpha |$ for the length of the path $\alpha$.  In addition, for a path $\alpha := e_1 \ldots e_n$, we define $s_\alpha := s_{e_1} \ldots s_{e_n}$.

A \emph{cycle} is a path $\alpha := e_1 \ldots e_n$ with $s(\alpha) = r(\alpha)$.  We say that a cycle $\alpha := e_1 \ldots e_n$ is a \emph{simple cycle} if $s(e_i) \neq s(e_1)$ for all $2 \leq i \leq n$.  An edge $f \in E^1$ is an \emph{exit} for the cycle $\alpha := e_1 \ldots e_n$ if there exists $i \in \{ 1, \ldots n \}$ such that $s(f) = s(e_i)$ and $f \neq e_i$.

A graph $E$ is \emph{downward directed} if for any $v, w \in E^0$ there exists $u \in E^0$ such that there is a path from $v$ to $u$ and there is a path from $w$ to $u$.

\begin{definition}
Let $E = (E^0, E^1, r, s)$ be a graph.  A subset $H \subseteq E^0$ is \emph{hereditary} if for all $e \in E^1$, $s(e) \in H$ implies $r(e) \in H$.  A hereditary subset $H$ is called \emph{saturated} if whenever $v \in E^0_\textnormal{reg}$ with $r(s^{-1}(v)) \subseteq H$, then $v \in H$.

One can verify that the intersection of any collection of saturated hereditary subsets is saturated hereditary, so for any hereditary subset $H \subseteq E^0$ we may define the \emph{saturation of $H$}, denoted $\overline{H}$ to be the smallest saturated hereditary subset of $E^0$ containing $H$; i.e., 
$$\overline{H} := \bigcap \{ S : \text{$S$ is saturated and hereditary subset of vertices and $H \subseteq S$} \}.$$
\end{definition}

\begin{remark}
Let $E = (E^0, E^1, r, s)$ be a graph.  If $H \subseteq E^0$ is a hereditary subset, we may recursively define $H_0 := H$ and $H_{n+1} := H_n \cup \{ v \in E^0_\textnormal{reg} : s(r^{-1}(v)) \subseteq H_n \}$ for all $n \in \N \cup \{ 0 \}$.  One can verify that $H_0 \subseteq H_1 \subseteq \ldots$ and $\overline{H} = \bigcup_{n=0}^\infty H_n$.  As an immediate consequence of this fact, we note that if $v \in \overline{H}$, then there exists a natural number $n \in \N$ such that all paths of length $n$ or greater with source $v$ have range in $H$.
\end{remark}

\begin{definition} \label{ideals-def}
If $E = (E^0, E^1, r, s)$ is a graph and $H$ is a saturated hereditary subset of $E$, the \emph{breaking vertices} of $H$ are the elements of the set
$$B_H := \{ v \in E^0_\textnormal{sing} : 0 < | s^{-1}(v) \cap r^{-1} (E^0 \setminus H) | < \infty \}.$$
For any $v \in B_H$ we define the \emph{gap projection} to be the element 
$$p_v^H := p_v - \sum_{e \in s^{-1}(v) \cap r^{-1} (E^0 \setminus H) } s_es_e^*.$$
An \emph{admissible pair} $(H,S)$ is a pair consisting of a saturated hereditary subset $H$ and a subset $S \subseteq B_H$.  For any admissible pair $(H,S)$ we let $\mathcal{I}_{(H,S)}$ denote the closed two-sided ideal of $C^*(E)$ generated by $\{ p_v : v \in H \} \cup \{ p_v^H : v \in S \}$, and for a fixed field $k$ we let $I_{(H,S)}$ denote the two-sided ideal of $L_k(E)$ generated by $\{ p_v : v \in H \} \cup \{ p_v^H : v \in S \}$.  Note, in particular, that we are using the italicized notation $\mathcal{I}_{(H,S)}$ to denote a closed two-sided ideal in $C^*(E)$, and we are using the non-italicized notation $I_{(H,S)}$ to denote a two-sided ideal in $L_k(E)$.  Furthermore, observe that when $k = \C$ and we make the identification $L_\C(E) \subseteq C^*(E)$, then $\mathcal{I}_{(H,S)}$ is the closure of $I_{(H,S)}$ in this situation.
\end{definition}

\begin{remark} \label{ideals-quotients-rem}
We use $$\mathbb{T} := \{ z \in \C : |z| = 1 \}$$ to denote the circle of unimodular complex numbers.  The following results are standard:

Let $E = (E^0, E^1, r, s)$ be a graph.  There exists a \emph{gauge action} $\gamma : \T \to \operatorname{Aut} C^*(E)$ with $\gamma_z(p_v) = p_v$ for all $v \in E^0$ and $\gamma_z(s_e) = zs_e$ for all $e \in E^1$.  In addition, for any field $k$ there is a $\mathbb{Z}$-grading on $L_k(E)$ with the $n$th homogeneous component given by 
$$L_k(E)_n := \operatorname{span} \{ s_\alpha s_\beta^* : \alpha, \beta \in E^* \text{ and } |\alpha|-|\beta| = n \}.$$

The assignment 
$$(H,S) \mapsto \mathcal{I}_{(H,S)}$$
is a bijection from the collection of saturated hereditary pairs of $E$ onto the gauge-invariant closed two-sided ideals of $C^*(E)$, and for any field $k$ the assignment
$$(H,S) \mapsto I_{(H,S)}$$
is a bijection from the collection of saturated hereditary pairs of $E$ onto the graded two-sided ideals of $L_k(E)$.

In addition, for any admissible pair $(H,S)$ we let $E \setminus (H,S)$ denote the graph whose vertex set is $$(E^0 \setminus H) \sqcup \{ w_v : v \in B_H \setminus S \}$$
and whose edge set is $$(E^1 \setminus r^{-1}(H)) \sqcup \{ f_e : e \in E^1 \text{ and } r(e) \in B_H \setminus H \}$$ with $r$ and $s$ extended to the edge set by letting $s(f_e) := s(e) \in E^0 \setminus H$ and $r(f_e) := w_{r(e)}$.

For any admissible pair $(H,S)$ in $E$ we have $C^*(E) / \mathcal{I}_{(H,S)} \cong C^*( E \setminus (H,S))$ and $L_k(E) / I_{(H,S)} \cong L_k( E \setminus (H,S))$.
\end{remark}

\begin{definition}
If $H$ is a hereditary (but not necessarily saturated) subset of $E^0$, for brevity we shall  let $I_H$ denote the two-sided ideal in $L_k(E)$ generated by $\{ p_v : v \in H \}$, and we let $\mathcal{I}_H$ denote the closed two-sided ideal in $C^*(E)$ generated by $\{ p_v : v \in H \}$.

If $H$ is hereditary and $\overline{H}$ is the saturation of $H$, then it is shown in \cite[Lemma~2.4.1]{AAM}  that
$$ I_H = I_{\overline{H}} =  I_{(\overline{H}, \emptyset)} :=  \operatorname{span}_k \{ s_\alpha s_\beta : \alpha, \beta \in E^* \text{ and } r(\alpha) = r(\beta) \in H \},$$
and it is shown in that \cite[Theorem~4.9, Remark~~4.12]{Rae} that
$$ \mathcal{I}_H = \mathcal{I}_{\overline{H}} =  \mathcal{I}_{(\overline{H}, \emptyset)} :=  \overline{\operatorname{span}} \{ s_\alpha s_\beta : \alpha, \beta \in E^* \text{ and } r(\alpha) = r(\beta) \in H \}.$$
Furthermore, note that when $k = \C$ and we make the identification $L_\C(E) \subseteq C^*(E)$, then for any hereditary subset $H$ we have $\mathcal{I}_H$ is the closure of $I_H$.
\end{definition}

\begin{definition}
Let $E = (E^0, E^1, r, s)$ be a graph.  A vertex $v \in E^0$ is called a \emph{bifurcation vertex} if $v$ is the source of two or more edges.  For any vertex $v \in E^0$, we define $T(v) := \{ w \in E^0 : \text{there is a path from $v$ to $w$} \}$.  Note that $T(v)$ is a hereditary subset of $E^0$.  We call a vertex $v \in E^0$ a \emph{line point} if $T(v)$ contains no bifurcating vertices and no vertex in $T(v)$ is the base point of a cycle.  (Observe that any sink is a line point.)  Note that if $v$ is a line point, then $v$ gives rise to a subgraph of $E$ of the form $v \to w_1 \to w_2 \to \cdots$ or of the form $v \to w_1 \to w_2 \to \cdots \to w_n$ for a sink $w_n$.  In particular, every line point is the source of a (finite or infinite) path with the property that every vertex of that path is also a line point.  
\end{definition}

\begin{definition}
Let $A$ be a $*$-algebra over a field $k$.  A collection of elements $\{ e_{i,j} \}_{i,j \in \Lambda}$ is called a set of \emph{matrix units} when for all $i,j,l,m \in \Lambda$ the elements in the collection satisfy
$$ e_{i,j} e_{l,m} := \begin{cases} e_{i,m} & \text{if $j=l$} \\ 0 & \text{if $j \neq l$} \end{cases} \qquad \text{ and } \qquad e_{i,j}^* = e_{j,i}.$$
It is straightforward to show that if one element in a set of matrix units is nonzero, then so is every other element in the set.  Hence we refer to a set of matrix units as a \emph{nonzero set of matrix units} if one (and hence every) element in the set in nonzero.
\end{definition}

\begin{remark} \label{matrix-units-rem}
If $\{ e_{i,j} \}_{i,j \in \Lambda}$ is a collection of nonzero matrix units in a $*$-algebra, the $*$-subalgebra generated by $\{ e_{i,j} \}_{i,j \in \Lambda}$ is isomorphic to the matrix algebra $M_\Lambda (k)$ of finitely-supported $\Lambda \times \Lambda$ matrices with entries in $k$.  If $\{ e_{i,j} \}_{i,j \in \Lambda}$ is a collection of matrix units in a $C^*$-algebra, the $C^*$-subalgebra generated by $\{ e_{i,j} \}_{i,j \in \Lambda}$ is isomorphic to the compact operators $\K(\Hi)$, where $\Hi := \ell^2 (\Lambda)$ is a Hilbert space whose dimension is equal to the cardinality of $\Lambda$.  (A proof of the $\K(\Hi)$ result can be found in \cite[Corollary~A.9 and Remark~A.10, p.103--104]{Rae}, and the $M_\Lambda(k)$ result can be found in \cite[Lemma 2.6.4]{AAM} and obtained by similar methods.)
\end{remark}

\begin{lemma} \label{ideal-of-line-point-matrix-and-compacts-lem}
Let $E = (E^0, E^1, r, s)$ be a graph and let $k$ be a field.  If $v \in E^0$ is a line point, then $H := T(v)$ is a hereditary subset.  Furthermore, if we let
$$ \Lambda := \{ e_1 \ldots e_n \in E^* : r(e_n) \in T(v) \text{ and } s(e_{n}) \notin T(v) \},$$
then $I_H \cong M_\Lambda (k)$ and $\mathcal{I}_H \cong \K( \ell^2(\Lambda))$.
\end{lemma}

\begin{proof}
The fact that $T(v)$ is hereditary follows from the definition of a line point.  Furthermore, the elements of $T(v)$ lie on a subgraph of the form
$$
\xymatrix{
v = w_0 \ar[r]^{e_1} & w_1 \ar[r]^{e_2} & w_2 \ar[r]^{e_3} & w_3 \ar[r]^{e_4} & \cdots 
}
$$
which is either an infinite path or a finite path ending at a sink.

To prove the result we shall produce a set of matrix units in $I_H$ and $\mathcal{I}_H$.  To do so, for each $i,j \in \N \cup \{ 0 \}$ with $i \leq j$, we define
$$ \mu (i, j) := e_{i+1} \ldots e_j$$
to be the unique path in $E$ from $w_i$ to $w_j$.  (Note that $\mu(i,i) = w_i$ when $i=j$.)  For each $\alpha, \beta \in \Lambda$ we define
$$e_{\alpha,\beta} := \begin{cases} s_{\alpha \mu(i, j)} s_\beta^*& \text{ if $r(\alpha) = w_i$ and $r(\beta) = w_j$ with $i \leq j$} \\
s_\alpha s_{\beta \mu(j,i)}^* & \text{ if $r(\alpha) = w_i$ and $r(\beta) = w_j$ with $j < i$.}
 \end{cases}$$
 It is straightforward to verify that $\{ e_{\alpha,\beta} : \alpha, \beta \in \Lambda \}$ is a nonzero set of matrix units.
 
We clearly have $\{ e_{\alpha,\beta} : \alpha, \beta \in \Lambda \} \subseteq \{ s_\alpha s_\beta^* : \alpha, \beta \in E^* \text{ and } r(\alpha) = r(\beta) \in T(v) \}$.
To establish the reverse inclusion, suppose $\alpha, \beta \in E^*$ and $r(\alpha) = r(\beta) \in T(v)$.  Then we may write $\alpha = \alpha' \lambda$ and $\beta = \beta' \rho$ with $\alpha', \beta' \in \Lambda$ and with $\lambda$ and $\rho$ paths along vertices in $T(v)$ with $r(\lambda) = r(\rho)$.  Since $\lambda$ and $\rho$ are paths along vertices in $T(v)$ with $r(\lambda) = r(\rho)$, either $\lambda$ extends $\rho$, or $\rho$ extends $\lambda$.  In the first case $\lambda = \mu(i,j) \rho$, so that
$$ s_\alpha s_\beta^* = s_{\alpha' \lambda} s_{\beta' \rho}^* = s_{\alpha' \mu(i,j) \rho} s_{\beta' \rho}^* =  s_{\alpha' \mu(i,j)} s_\rho s_\rho^* s_{\beta'}^* = s_{\alpha' \mu(i,j)} s_{\beta'}^* = e_{\alpha',\beta'}.
$$
In the second case $\rho = \mu(j,i) \lambda$, so that
$$ s_\alpha s_\beta^* = s_{\alpha' \lambda} s_{\beta' \rho}^* = s_{\alpha' \lambda} s_{\beta' \mu(j,i) \lambda}^* =  s_{\alpha'} s_\lambda s_\lambda^* s_{\beta'\mu(j,i)}^* = s_{\alpha'} s_{\beta' \mu(j,i)}^* = e_{\alpha',\beta'}.
$$
Thus 
$$\{ e_{\alpha,\beta} : \alpha, \beta \in \Lambda \} = \{ s_\alpha s_\beta^* : \alpha, \beta \in E^* \text{ and } r(\alpha) = r(\beta) \in T(v) \}.$$
It follows that
$$I_H  = \operatorname{span}_k \{ s_\alpha s_\beta^* : \alpha, \beta \in E^* \text{ and } r(\alpha) = r(\beta) \in H \} =  \operatorname{span}_k \{ e_{\alpha, \beta} : \alpha, \beta \in \Lambda \}$$
and
$$\mathcal{I}_H  = \overline{\operatorname{span}} \{ s_\alpha s_\beta^* : \alpha, \beta \in E^* \text{ and } r(\alpha) = r(\beta) \in H \} =  \overline{\operatorname{span}} \{ e_{\alpha, \beta} : \alpha, \beta \in \Lambda \}.$$
Since $\{ e_{\alpha, \beta} : \alpha, \beta \in \Lambda \}$ is a nonzero set of matrix units, it follows (see Remark~\ref{matrix-units-rem}) that $I_H \cong M_\Lambda (k)$ and $\mathcal{I}_H \cong K(\ell^2 (\Lambda))$.
\end{proof}

\section{Irreducible Representations of Graph Algebras} \label{Irreducible-repns-sec}

In this section we construct representations of Leavitt path algebras and graph $C^*$-algebras on vector spaces and Hilbert spaces determined by paths in the graph.  Similar representations for Leavitt path algebras have been considered by Chen \cite[Section~3]{Che}, and similar representations for graph $C^*$-algebras have been considered by Carlsen and Sims \cite{CS}.  In his analysis, Chen only considered representations of Leavitt path algebras coming from infinite paths or from finite paths ending at sinks, but his techniques were later extended to produce representations coming from finite paths ending at singular vertices in \cite{AR14} and \cite{Ran}.  On the $C^*$-algebra side, Carlsen and Sims have only considered $*$-representations of graph $C^*$-algebras coming from infinite paths.  For our purposes, we will also need $*$-representations of graph $C^*$-algebras coming from infinite finite paths that end at a singular vertex (i.e., a sink or infinite emitter).  We develop these results in this section, guided by the results of \cite{Che}, \cite{AR14}, and \cite{Ran}.  For the benefit of our $C^*$-algebraist readers, as we develop the $*$-representation results for graph $C^*$-algebras we also give a presentation (with proofs) of the representation results for Leavitt path algebras, phrased in a manner similar to analytic results, 
to highlight the parallels between the analytic and algebraic theories.

We mention that our results in Theorem~\ref{decomposition-thm}(a) recover Chen's results in the case that the graph has no infinite emitters.  Likewise, our results in Theorem~\ref{decomposition-thm}(b) recover Carlsen and Sims' results in the case the graph has no sinks and no infinite emitters.

\begin{definition}[Representations and $*$-Representations] $ $
\begin{itemize}
\item[(i)] (\textsc{$k$-algebras}) Let $k$ be a field.  Given a $k$-algebra $A$, a \emph{representation} of $A$ consists of a $k$-vector space $V$ together with a $k$-algebra homomorphism $\rho : A \to \End_k (V)$.  We say the representation $\rho : A \to \End_k (V)$ is \emph{faithful} when $\rho$ is injective.  If $W$ is a $k$-vector subspace of $V$, we say that $W$ is $\rho$-invariant if $\rho(a)(W) \subseteq W$ for all $a \in A$, and we say that $\rho$ is \emph{irreducible} when the only $\rho$-invariant subspaces of $W$ are $\{ 0 \}$ and $W$.  In addition, two representations $\rho_1 : A \to \End (V_1)$ and $\rho_2 : A \to \End (V_2)$ are defined to be \emph{algebraically equivalent} if there exists a $k$-linear isomorphism $T : V_1 \to V_2$ such that 
$$ \rho_1(a) = T^{-1} \circ \rho_2(a) \circ T \quad \text{ for all $a \in A$.}$$
Observe that if $\rho : A \to \End_k (V)$ is a representation of $A$, then $V$ has the structure of an $A$-module by defining $a \cdot x := \rho(a)(x)$ for $a \in A$ and $x \in V$.  Moreover, every $A$-module arises in this way.  In addition, this $A$-module is a faithful module if and only if $\rho$ is a faithful representation, and is simple (i.e., has no submodules other than $\{ 0 \}$ and $V$) if and only if $\rho$ is an irreducible representation.

\item[(ii)] (\textsc{$C^*$-algebras}) Given a $C^*$-algebra $A$, a \emph{$*$-representation} of $A$ consists of a Hilbert space $\Hi$ together with a $*$-homomorphism $\pi : A \to \B (\Hi )$.  We say the $*$-representation $\pi : A \to  \B (\Hi )$ is \emph{faithful} when $\pi$ is injective.  If $\mathcal{K}$ is a closed subspace of $\Hi$, we say that $\mathcal{K}$ is $\pi$-invariant if $\pi(a)(\mathcal{K}) \subseteq \mathcal{K}$ for all $a \in A$, and we say that $\pi$ is \emph{irreducible} when the only closed $\pi$-invariant subspaces of $\Hi$ are $\{ 0 \}$ and $\Hi$.  In addition, two $*$-representations $\pi_1 : A \to \B(\Hi_1)$ and $\pi_2 : A \to \B(\Hi_2)$ are defined to be \emph{unitarily equivalent} if there exists a unitary $U : \Hi_1 \to \Hi_2$ such that 
$$ \pi_1(a) = U^* \circ \pi_2(a) \circ U \quad \text{ for all $a \in A$.}$$
\end{itemize}

\end{definition}

\begin{remark} \label{algebraic-unitary-same-rem}
Note that a $*$-representation of a $C^*$-algebra is also a $\C$-algebra representation, and thus it makes sense to discuss both algebraic equivalence and unitary equivalence of $*$-representations.  Furthermore, given a $C^*$-algebra $A$, one might be inclined to define two $*$-representations $\pi_1 : A \to \B(\Hi_1)$ and $\pi_2 : A \to \B(\Hi_2)$ to be ``equivalent" (in analogy with equivalence of algebra representations) if there exists a bounded (i.e., continuous) isomorphism $T: \Hi_1 \to \Hi_2$ such that $$ \pi_1(a) = T^{-1} \circ \pi_2(a) \circ T \quad \text{ for all $a \in A$.}$$  Such a $T$ is called an ``intertwining isomorphism".  However, it is well-known that two  $*$-representations of a $C^*$-algebra are ``equivalent" via an intertwining isomorphism if and only if they are unitarily equivalent (see \cite[Ch.2, \S 2.2.2, p.32]{Dix} for a proof of this fact).  Thus for representations of $C^*$-algebras, the definition of unitarily equivalent has emerged as the most useful way to characterize the appropriate definition of ``equivalence".  Also note that if two $*$-representations of a $C^*$-algebra are unitarily equivalent, then they are algebraically equivalent.  However, the converse implication does not hold in general since the intertwining isomorphism in an algebraic equivalence may not be bounded.
\end{remark}

For any vertex $v \in E^0$ we let $ E^*(v) := \{ \alpha \in E^* : r(\alpha) = v \}$ denote the set of paths with range $v$.  We say that a path $\alpha \in E^*$ is a \emph{singular path} if $r(\alpha) \in E^0_\textnormal{sing}$, and we let
$$E^*_\textnormal{sing} := \bigcup_{v_0 \in E^0_\textnormal{sing}} E^*(v_0)$$
denote the set of singular paths in $E$.

We define an \emph{infinite path} in $E$ to be a sequence of edges $\omega := e_1 e_2 \ldots$ with $r(e_i) = s(e_{i+1})$ for all $i \in \N$.  (Note that, strictly speaking, an ``infinite path" is not a ``path".)  We define the \emph{infinite path space} $E^\infty$ to be the set of infinite paths in $E$.  We extend the source map $s : E^1 \to E^0$ to $E^\infty$ as follows:  If $\omega := e_1 e_2 \ldots \in E^\infty$, then $s (\omega) := s(e_1)$.  We also define the length of an infinite path to be infinity.

\begin{definition} \label{boundary-paths-def}
The \emph{boundary of $E$} is defined to be the set
$$ \partial E := E^\infty \sqcup E^*_\textnormal{sing}$$
and we call elements of $\partial E$ \emph{boundary paths}.  Note that boundary paths in $\partial E$ are of two mutually exclusive types: infinite paths of $E^\infty$ and singular paths of $E^*_\textnormal{sing}$.
\end{definition}

\begin{definition}[The Shift Map]
For any graph $E$ we define the \emph{shift map} $$\sigma_E : \partial E \to \partial E$$ 
as follows:  Let $\alpha \in \partial E$.  If $\alpha := v$ is a vertex, then $\sigma_E (\alpha) = v$.  If $\alpha := e_1 \ldots e_n$ is a finite path of positive length, then $\sigma_E (\alpha) := e_2 \ldots e_n$.  If $\alpha := e_1 e_2 \ldots$ is an infinite path, then $\sigma_E (\alpha) := e_2 e_3 \ldots$.  For every natural number $n \in \N$ we also let $\sigma_E^n := \sigma_E \circ \ldots \circ \sigma_E$ denote the $n$-fold composition of $\sigma_E$.

Note that $\sigma_E (E^\infty) \subseteq E^\infty$ and $\sigma_E (E^*(v_0)) \subseteq E^*(v_0)$ for all $v_0 \in E^0_\textnormal{sing}$.
\end{definition}

\begin{definition}
If $\alpha, \beta \in  \partial E$, we say $\alpha$ is \emph{shift-tail equivalent} to $\beta$, written $\alpha \sim_{st} \beta$, if there exist $m,n \in \N$ such that $\sigma_E^m (\alpha) = \sigma_E^n (\beta)$.  It is straightforward to verify that shift-tail equivalence is an equivalence relation on $\partial E$.

For $\alpha \in \partial E$, we let $$[\alpha ] := \{ \beta \in \partial E : \beta \sim_{st} \alpha \}$$
denote the shift-tail equivalence class of $\alpha$ in $\partial E$.  In addition, we let $$\widetilde{\partial} E := \{ [\alpha] : \alpha \in \partial E \}$$ denote the collection of shift-tail equivalence classes of paths in $\partial E$.

\end{definition}

\begin{lemma} \label{st-basic-lem}
Let $E$ be a graph, and let $\alpha, \beta \in \partial E$.
\begin{itemize}
\item[(i)] If $\alpha \sim_{st} \beta$, then either $\alpha$ and $\beta$ are both infinite paths or $\alpha$ and $\beta$ are both finite paths.  (In particular, if $\alpha \in E^\infty$, then $[\alpha] \subseteq E^\infty$, and if $\alpha \in E^*_\textnormal{sing}$, then $[\alpha] \subseteq E^*_\textnormal{sing}$.)
\item[(ii)] If $\alpha, \beta \in E^*_\textnormal{sing}$, then $\alpha \sim_{st} \beta$ if and only $r(\alpha) = r(\beta)$.  (In particular, if $\alpha \in E^*_\textnormal{sing}$ and $r(\alpha) = v_0$, then $[ \alpha ] = E^*( v_0)$.)
\item[(iii)] If $\alpha, \beta \in E^\infty$, then $\alpha \sim_{st} \beta$ if and only if $\alpha = \mu \gamma$ and $\beta = \nu \gamma$ for some $\mu, \nu \in E^*$ and some $\gamma \in E^\infty$.
\end{itemize}

\end{lemma}

\begin{proof}
(i) Simply observe that for any $n \in \mathbb{N}$ the $n$-fold composition $\sigma_E^n$ takes finite paths to finite paths and infinite paths to infinite paths.  Thus, in order for two boundary paths $\alpha, \beta \in \partial E$ to satisfy $\sigma_E^m (\alpha) = \sigma_E^n (\beta)$ with $m,n \in \N$, we must have that $\alpha$ and $\beta$ are either both infinite paths or both finite paths.  

(ii) If $\alpha \in E^*_\textnormal{sing}$, then for any $n \in \N$ we see that $\sigma_E^n (\alpha)$ has the same range as $\alpha$.  In addition, for $n$ greater than or equal to the length of $\alpha$ we have $\sigma_E^n (\alpha) = r(\alpha)$.  Thus, for two boundary paths $\alpha, \beta \in E^*_\textnormal{sing}$, we have $\sigma_E^m (\alpha) = \sigma_E^n (\beta)$ for some $m,n \in \N$ if and only if $r(\alpha) = r(\beta)$.

(iii)  If $\alpha$ is an infinite path, then $\sigma^n(\alpha) = \gamma$ if any only if there exists a finite path $\mu$ of length $n$ for which $\alpha = \mu \gamma$.  Thus when $\alpha, \beta \in E^\infty$ we have $\sigma_E^m (\alpha) = \sigma_E^n (\beta) = \gamma$ for $m,n \in \N$ if and only if $\alpha = \mu \gamma$ for a finite path $\mu$ (of length $m$) and $\beta = \nu \gamma$ for a finite path $\nu$ (of length $n$).
\end{proof}

\begin{definition}[Vector spaces and Hilbert spaces associated with graphs]
Fix a field $k$ and a graph $E := (E^0, E^1, r, s)$.  We define $$V_{(\partial E,k)} :=  \operatorname{span}_k \partial E$$ to be the vector space over $k$ with basis $\partial E$.  We also define $$\Hi_{\partial E} := \ell^2 ( \partial E)$$ to be the complex Hilbert space with orthonormal basis $\partial E$.  We note that $V_{(\partial E,\C)}$ is a dense subspace of $\Hi_{\partial E}$.

For $\alpha \in \partial E$, we define $$V_{ ([\alpha],k) } := \operatorname{span}_k [\alpha]$$ to be the vector space over $k$ with basis $[\alpha]$ and we define $$\Hi_{ [\alpha] }:= \ell^2 ( [\alpha])$$ to be the complex Hilbert space with orthonormal basis $[\alpha]$.  For each $\alpha \in \partial E$ we see that $V_{ ([\alpha],\C) }$ is a dense subspace of $\Hi_{ [\alpha] }$.  In addition, we observe that $$V_{(\partial E,k)} = \bigoplus_{ [\alpha] \in \widetilde{\partial} E } V_{ ([p], k) } \qquad
\text{ and } \qquad \Hi_{\partial E} = \bigoplus_{ [\alpha] \in \widetilde{\partial} E} \Hi_{ [\alpha] }.$$
Furthermore, if $\alpha \in E^*_\textnormal{sing}$, then Lemma~\ref{st-basic-lem}(b) implies $V_{ ( [\alpha], k )} = \operatorname{span}_k E^*(r(\alpha))$ and $\Hi_{[\alpha]} = \ell^2 ( E^*(r(\alpha)) )$.
\end{definition}

Given a graph $E$, we shall now construct an algebra representation of the Leavitt path algebra $L_k(E)$ on the vector space $V_{(\partial E,k)}$ and a $*$-representation of the graph $C^*$-algebra $C^*(E)$ on the Hilbert space $\Hi_{\partial E}$.

\begin{proposition} \label{representation-exists-prop}
Let $E = (E^0, E^1, r, s)$ be a graph.  
\begin{itemize}
\item[(a)] For any field $k$ there exists an algebra representation
$$\rho_{E,k} : L_k(E) \to \End_k V_{ ( \partial E ,k) }$$ with
$$\rho_{E,k} (s_e) \alpha = \begin{cases} e\alpha & \text{ if $s(\alpha) = r(e)$} \\ 0 & \text{ if $s(\alpha) \neq r(e)$,} \end{cases}
\qquad
\rho_{E,k} (s_e^*) \alpha = \begin{cases} \alpha' & \text{ if $\alpha = e \alpha'$} \\ 0 & \text{ otherwise,} \end{cases}
$$
and
$$\rho_{E,k} (p_v) \alpha = \begin{cases} \alpha & \text{ if $s(\alpha) = v$} \\ 0 & \text{ if $s(\alpha) \neq v$} \end{cases}$$
for all $e \in E^1$, $v \in E^0$, and $\alpha \in \partial E$.
\item[(b)]  There exists a $*$-representation
$$\pi_E : C^*(E) \to \B ( \Hi_{ \partial E})$$ with
$$\pi_E (s_e) \alpha = \begin{cases} e\alpha & \text{ if $s(\alpha) = r(e)$} \\ 0 & \text{ if $s(\alpha) \neq r(e)$,} \end{cases}
\qquad
\pi_E (s_e^*) \alpha = \begin{cases} e\alpha & \text{ if $\alpha = e \alpha'$} \\ 0 & \text{ otherwise,} \end{cases}
$$
and
$$\pi_E (p_v) \alpha = \begin{cases} \alpha & \text{ if $s(\alpha) = v$} \\ 0 & \text{ if $s(\alpha) \neq v$} \end{cases}$$
for all $e \in E^1$, $v \in E^0$, and $\alpha \in \partial E$.
\end{itemize}
\end{proposition}

\begin{proof}
For (a), observe that the displayed equations determine a set of linear transformations $\{ \rho_{E,k} (s_e), \rho_{E,k} (s_e^*), \rho_{E,k} (p_v) : e \in E^1, v \in E^0 \}$ in $\End_k V_{ ( \partial E ,k) }$.  Some straightforward but slightly tedious calculations allow one to verify that the elements of $\{ \rho_{E,k} (s_e), \rho_{E,k} (s_e^*), \rho_{E,k} (p_v) : e \in E^1, v \in E^0 \}$ satisfy the defining relations for the Leavitt path algebra, and hence by the universal property of $L_k(E)$ there exists a $k$-algebra homomorphism $\rho_{E,k} : L_k(E) \to \End_k V_{ ( \partial E ,k) }$ taking $s_e \mapsto \rho_{E,k} (s_e)$, $s_e^* \mapsto \rho_{E,k} (s_e^*)$, and $p_v \mapsto \rho_{E,k} (p_v)$.

For (b), similar calculations show the displayed equations determine a set of bounded linear operators $\{ \pi_E (s_e), \pi_E (p_v) : e \in E^1, v \in E^0 \}$ in $B ( \Hi_{ \partial E})$.  Straightforward calculations show the elements of $\{ \pi_E (s_e), \pi_E(E) (p_v) : e \in E^1, v \in E^0 \}$ are a Cuntz-Krieger $E$-family, and hence by the universal property of $C^*(E)$ there exists a $*$-homomorphism $\pi_E : C^*(E) \to \B ( \Hi_{ \partial E})$ taking $s_e \mapsto \pi_E (s_e)$ and $p_v \mapsto \pi_E (p_v)$.  Furthermore, one can verify that the adjoint $\pi_E (s_e)^*$ is equal to the operator $\pi_E (s_e^*)$ defined above, and since $\pi$ is a $*$-homomorphism we conclude that $\pi$ takes $s_E^* \mapsto \pi_E (s_e)^* = \pi(s_e^*)$.
\end{proof}

\begin{remark}[Decomposition of a representation as a direct sum of irreducible representations]
Let $A$ be a $k$-algebra and $\rho :A \to \End_k (V)$ a representation of $A$ on the $k$-vector space $V$.  If $W$ is a $\rho$-invariant subspace of $V$ (i.e., $\rho(a)(w) \in W$ for all $a \in A$ and $w \in W$), then for each $a \in A$ the endomorphism $\rho(a) : V \to V$ restricts to an endomorphism $\rho(a) |_W : W \to W$.  Consequently, we obtain a representation $\rho |_W : A \to \End_k (W)$ defined by $\rho  |_W (a) := \rho(a) |_W$.  We call $\rho|_W$ the \emph{subrepresentation} of $\rho$ obtained by restricting to $W$.

In addition, if $A$ is a $k$-algebra, $I$ is an indexing set, and for each $i \in I$ we have a representation $\rho_i :A \to \End_k (V_i)$ on a $k$-vector space $V_i$, we may define the direct sum representation
$ \bigoplus_{i \in I} \rho_i : A \to \End_k \left( \bigoplus_{i \in I} V_i \right)$ by 
$$\left( \bigoplus_{i \in I} \rho_i \right) (a) \left(  (v_i)_{i \in I} \right) := \bigoplus_{i \in I} \rho_i (a) (v_i).$$

Given an algebra representation $\rho : A \to \End_k (V)$ it is desirable (if possible) to decompose $V$ as a direct sum of $\rho$-invariant subspaces; that is, write $V = \bigoplus_{i \in I} W_i$ with each subspace in the collection $\{ W_i : i \in I \}$ $\rho$-invariant.  Then, for each $i \in I$, our representation $\rho$ restricts to a representation $\rho |_{W_i} : A \to \End_k(W_i)$, and furthermore $\rho = \bigoplus_{i \in I} \rho |_{W_i} : A \to \End_k \left( \bigoplus_{i \in I} W_i \right)$.  When such a decomposition exists, one can study $\rho$ in terms of the irreducible subrepresentations $\rho |_{W_i}$.

Similar statements hold for $C^*$-algebras:  Given a $*$-representation $\pi : A \to \B(\Hi)$, we seek (if possible) to decompose $\Hi$ as a (Hilbert space) direct sum $\Hi = \bigoplus_{i \in I} \Hi_i$ with each subspace in the collection $\{ \Hi_i : i \in I \}$ closed and $\pi$-invariant.  Then $\pi = \bigoplus_{i \in I} \pi |_{\Hi_i} : A \to \B ( \bigoplus_{i \in I} \Hi_i )$.
\end{remark}

The following results allow us to decompose the representations of Proposition~\ref{representation-exists-prop} as direct sums of irreducible representations, which we describe in Theorem~\ref{decomposition-thm}.

\begin{lemma} \label{inv-subspace-with-basis-elt-all-lem} 
Let $E$ be a graph, let $k$ be a field, and let $\rho_{E,k} : L_k(E) \to \End_k V_{ ( \partial E ,k) }$ and $\pi_E : C^*(E) \to \B ( \Hi_{ \partial E})$ be as defined in Proposition~\ref{representation-exists-prop}.
\begin{itemize}
\item[(a)] Let $\alpha \in \partial E$.  If $W$ is a $\rho_{E,k}$-invariant subspace of $V_{ ([\alpha],k) } := \operatorname{span}_k [\alpha]$ and there exists $\beta \in W$ for some $\beta \in [\alpha]$, then $W = V_{ ( [\alpha], k ) }$.
\item[(b)] Let $\alpha \in \partial E$. If $\mathcal{K}$ is a closed $\pi_E$-invariant subspace of $\Hi_{  [\alpha] } := \ell^2( [\alpha ] )$ and there exists $\beta \in \mathcal{K}$ for some $\beta \in [\alpha]$, then $\mathcal{K} = \Hi_{  [\alpha] }$.
\end{itemize}
\end{lemma}

\begin{proof}
For (a) choose any $\gamma \in [\alpha]$.  Since $\gamma \sim_{st} \alpha \sim_{st} \beta$, it follows that $\gamma = \mu \delta$ and $\beta = \nu \delta$ for some $\mu, \nu \in E^*$ and $\delta \in \partial E$.  Then, since $W$ is $\rho_{E,k}$-invariant, we have
$$ \gamma = \mu \delta = \rho_{E,k} (s_\mu s_\nu^*) \nu \delta = \rho_{E,k} (s_\mu s_\nu^*) \beta \in W.$$
Hence $[\alpha] \subseteq W$ and since $W$ is a subspace, $V_{ ( [\alpha], k )} := \operatorname{span}_k [\alpha ] \subseteq W$.  Thus $V_{ ( [\alpha], k )} = W$.

For (b) choose any $\gamma \in [\alpha]$.  Since $\gamma \sim_{st} \alpha \sim_{st} \beta$, it follows that $\gamma = \mu \delta$ and $\beta = \nu \delta$ for some $\mu, \nu \in E^*$ and $\delta \in \partial E$.  Then, since $\mathcal{K}$ is $\pi_E$-invariant, we have
$$ \gamma = \mu \delta = \pi_E (s_\mu s_\nu^*) \nu \delta = \pi_E (s_\mu s_\nu^*) \beta \in  \mathcal{K}$$
Hence $[\alpha] \subseteq \mathcal{K}$ and since $\mathcal{K}$ is a closed subspace, $\Hi_{ [\alpha] } := \overline{\operatorname{span}} \, [\alpha ] \subseteq \mathcal{K}$.  Thus $\Hi_{  [\alpha] } = \mathcal{K}$.
\end{proof}

\begin{proposition}  \label{restriction-irreducible-prop} 
Let $E$ be a graph, let $k$ be a field, and let $\rho_{E,k} : L_k(E) \to \End_k V_{ ( \partial E ,k) }$ and $\pi_E : C^*(E) \to \B ( \Hi_{ \partial E})$ be as defined in Proposition~\ref{representation-exists-prop}.
\begin{itemize}
\item[(a)] For any $\alpha \in \partial E$ the subspace $V_{ ([\alpha],k) } := \operatorname{span}_k [\alpha]$ is $\rho_{E, k}$-invariant, and the restriction $\rho_{E, k} |_{ V_{ ([\alpha], k)} } : L_k(E) \to \End_k V_{ ([\alpha],k) }$ is an irreducible algebra representation.
\item[(b)] For any $\alpha \in \partial E$ the subspace $\Hi_{ [\alpha] } := \ell^2( [\alpha ] )$ is $\pi_E$-invariant, and the restriction $\pi_E |_{ \Hi_{ [\alpha] } } : C^*(E) \to \B ( \Hi_{ [\alpha] } )$ is an irreducible $*$-representation.
\end{itemize}
\end{proposition}

\begin{proof}
We shall first show $V_{ ([\alpha], k)}$ is $\rho_{E, k}$-invariant.  Choose any basis element $\beta \in [\alpha ]$.  Then $\beta \sim_{st} \alpha$.  For any finite paths $\mu, \nu \in E^*$, we have $\rho_{E, k} (s_\mu s_\nu^*) \beta = 0$ unless $\beta = \nu \gamma$ for some $\gamma \in \partial E$ with $s(\gamma) = r(\mu)$, in which case $\rho_{E, k} (s_\mu s_\nu^*) \beta =  \mu \gamma$.  Since
$$ \mu \gamma \sim_{st} \gamma \sim_{st} \beta \sim_{st} \alpha$$
we conclude that in this case $\mu \gamma \in V_{ ([\alpha], k)}$.  Thus $\rho_{E, k} (s_\mu s_\nu^*) \beta \in V_{ ([\alpha], k)}$ for all $\mu, \nu \in E^*$, and by linearity $\rho_{E, k} (a) \beta \in V_{ ([\alpha], k)}$ for all $a \in L_k(E)$.  Applying linearity again, we conclude $\rho_{E, k} (a) ( V_{ ([\alpha], k)} ) \subseteq V_{ ([\alpha], k)}$ for all $a \in L_k(E)$.  Thus $V_{ ([\alpha], k)}$ is $\rho_{E, k}$-invariant.  This establishes the claim of invariance in (a).

A similar argument together with an application of the continuity of $\pi_E$ shows that $\pi_E(a) ( H_{ [\alpha] } ) \subseteq \Hi_{ [\alpha] }$ for all $a \in C^*(E)$.  Thus $\Hi_{ [\alpha] }$ is $\pi_E$-invariant, establishing the claim of invariance in (b).

Next, we shall show the restriction $\rho_{E, k} |_{ V_{ ([\alpha], k)} } $ is irreducible.   Consider two cases:

\medskip

\noindent \textsc{Case I:}  $\alpha \in E^\infty$. \\
Let $W$ be a nonzero $\rho_{E, k}$-invariant subspace of $V_{ ([\alpha], k)}$, and choose a nonzero element $\vec{w} \in W$.  Write
$$\vec{w} = \sum_{i=1}^n z_i \beta_i$$
with the $\beta_i$'s distinct infinite paths in $[\alpha]$ and  $z_i \in k \setminus \{ 0 \}$ for all $1 \leq i \leq n$.  For each $1 \leq i \leq n$, write the infinite path $\beta_i$ as 
$$ \beta_i = \beta_{i,1} \beta_{i,2} \beta_{i,3} \ldots \qquad \text{ with $\beta_{i,j} \in E^1$ for all $j \in \N$}.$$
Since the $\beta_i$'s are distinct infinite paths, we may choose $N \in \N$ such that 
$$\mu := \beta_{1, 1} \beta_{1,2}  \ldots \beta_{1, N}$$ is distinct from $\beta_{i,1} \beta_{i,2} \ldots \beta_{i, N}$ for all $2 \leq i \leq n$.  Thus $\rho_{E, k} (s_\mu s_\mu^*) \beta_i = 0$ for all $2 \leq i \leq n$.  Hence
\begin{align*}
z_1 \beta_1 =  z_1 \rho_{E, k} (s_\mu s_\mu^*) \beta_1  &= \sum_{i=1}^n z_i \rho_{E, k} (s_\mu s_\mu^*) \beta_i \\
&=  \rho_{E, k} (s_\mu s_\mu^*) \left( \sum_{i=1}^n z_i \beta_i \right) =   \rho_{E, k} (s_\mu s_\mu^*) \vec{w} \in W.
\end{align*}
Since $z_1 \neq 0$ and $W$ is a subspace, we conclude $\beta_1 \in W$.  By Lemma~\ref{inv-subspace-with-basis-elt-all-lem}(a) it follows that $W = V_{([\alpha], k)}$.  Hence $\rho_{E, k} |_{ V_{ ([\alpha], k|} } $ is an irreducible algebra representation.

\medskip

\noindent \textsc{Case II:} $\alpha \in E^*_\textnormal{sing}$. \\
In this case $[\alpha]$ consists of finite paths all ending at a particular singular vertex $v_0$.  Let $W$ be a nonzero $\rho_{E, k}$-invariant subspace of $V_{ ([\alpha], k)}$, and choose a nonzero element $\vec{w} \in W$.  Write
$$ \vec{w} = \sum_{i=1}^n z_i \beta_i$$
with the $\beta_i$'s distinct elements of $[\alpha]$ and $z_i \in k \setminus \{ 0 \}$ for all $1 \leq i \leq n$.    Furthermore, after possibly reordering and relabeling the subscripts, assume that $\beta_1$ has maximal length among the $\beta_i$'s; that is, $|\beta_1| \geq | \beta_i |$ for all $1 \leq i \leq n$.  Due to the maximality of $\beta_1$ and the fact the $\beta_i$'s are distinct, no $\beta_i$ with $2 \leq i \leq n$ can extend $\beta_1$, and hence $\rho_{E, k} (s_{\beta_1}^*) \beta_i = 0$ for all $2 \leq i \leq n$.  Thus
\begin{align*}
z_1  v_0 = z_1 \rho_{E, k} (s_{\beta_1}^*) \beta_1 = \sum_{i=1}^n z_i &\rho_{E, k} (s_{\beta_1}^*) \beta_i \\
&= \rho_{E, k} (s_{\beta_1}^*) \left( \sum_{i=1}^n z_i \beta_i \right) = \rho_{E, k} (s_{\beta_1}^*) \vec{w} \in W.
\end{align*}
Since $z_i \neq 0$ and $W$ is a subspace, we conclude $\beta_1 \in W$.  By Lemma~\ref{inv-subspace-with-basis-elt-all-lem}(a) it follows that $W = V_{([\alpha], k)}$.  Hence $\rho_{E, k} |_{ V_{ ([\alpha], k|} } $ is an irreducible algebra representation.

The two cases above show that (a) holds.  We have already shown the $\pi_E$-invariance of $\Hi_{[\alpha]}$ claimed in (b).  Thus, to establish (b), it remains to show the restriction $\pi_E |_{ \Hi_{ [\alpha] } } $ is irreducible.  Consider two cases:

\medskip

\noindent \textsc{Case I:}  $\alpha \in E^\infty$. \\
Let $\mathcal{K}$ be a closed nonzero $\pi_E$-invariant subspace of $\Hi_{ [\alpha] }$, and choose a nonzero element $\vec{k} \in \mathcal{K}$.  Write
$$\vec{k} = \sum_{i=1}^\infty z_i \beta_i$$
with the $\beta_i$'s distinct infinite paths in $[\alpha]$ and  $z_1 \in \C \setminus \{ 0 \}$.

Let $\epsilon > 0$.  Choose $n \in \N$ such that $$\left\| \sum_{i=n+1}^\infty z_i \beta_i \right\| < \epsilon.$$  For each $i \in \N$, write the infinite path $\beta_i$ as 
$$ \beta_i = \beta_{i,1} \beta_{i,2} \beta_{i,3} \ldots \qquad \text{ with $\beta_{i,j} \in E^1$ for all $j \in \N$.}$$
Since the $\beta_i$'s are distinct infinite paths, we may choose $N \in \N$ such that 
$$\mu := \beta_{1, 1} \beta_{1,2}  \ldots \beta_{1, N}$$ is distinct from $\beta_{i,1} \beta_{i,2} \ldots \beta_{i, N}$ for all $2 \leq i \leq n$.  Thus $\pi_E (s_\mu s_\mu^*) \beta_i = 0$ for all $2 \leq i \leq n$.  Hence
\begin{align*}
\pi_E (s_\mu s_\mu^*) \vec{k} &= \pi_E (s_\mu s_\mu^*) \left( \sum_{i=1}^\infty z_i \beta_i  \right) \\ 
&= \pi_E (s_\mu s_\mu^*) \left( \sum_{i=1}^n z_i \beta_i  \right) + \pi_E (s_\mu s_\mu^*) \left( \sum_{i=n+1}^\infty z_i \beta_i  \right) \\
&= \sum_{i=1}^n z_i \pi_E (s_\mu s_\mu^*) \beta_i  + \pi_E (s_\mu s_\mu^*) \left( \sum_{i=n+1}^\infty z_i \beta_i  \right) \\
&= z_1\beta_1  + \pi_E (s_\mu s_\mu^*) \left( \sum_{i=n+1}^\infty z_i \beta_i  \right).
\end{align*}
Since $\pi_E (s_\mu s_\mu^*)$ is a partial isometry, we have $\| \pi_E (s_\mu s_\mu^*) \| = 1$.  Thus
\begin{align*}
\left\| \pi_E (s_\mu s_\mu^*) \vec{k} - z_1\beta_1 \right\| &= \left\| \pi_E (s_\mu s_\mu^*) \left( \sum_{i=n+1}^\infty z_i \beta_i  \right) \right\| \\
&\leq  \left\| \pi_E (s_\mu s_\mu^*) \right\| \cdot \left\| \left( \sum_{i=n+1}^\infty z_i \beta_i  \right) \right\| =  \left\| \left( \sum_{i=n+1}^\infty z_i \beta_i  \right) \right\| < \epsilon.
\end{align*}
Because $\mathcal{K}$ is $\pi_E$-invariant, $\pi_E (s_\mu s_\mu^*) \vec{k} \in \mathcal{K}$.  In addition, since $\epsilon > 0$ was arbitrary, $z_1\beta_1$ is in the closure of $\mathcal{K}$.  Since $\mathcal{K}$ is closed, $z_1\beta_1 \in \mathcal{K}$, and because $z_i \neq 0$ and $\mathcal{K}$ is a subspace, we conclude $\beta_1 \in \mathcal{K}$.  By Lemma~\ref{inv-subspace-with-basis-elt-all-lem}(b) it follows that $\mathcal{K} = \Hi_{[\alpha] }$.  Hence $\pi_E |_{ \Hi_{ [\alpha] } }$ is an irreducible algebra representation.

\noindent \textsc{Case II:}  $\alpha \in E^*_\textnormal{sing}$. \\
In this case $[\alpha]$ consists of finite paths all ending at a particular singular vertex $v_0$.  Let $\mathcal{K}$ be a closed nonzero $\pi_E$-invariant subspace of $H_{ [\alpha] }$, and choose a nonzero element $\vec{k} \in \mathcal{K}$.  Write
$$ \vec{k} = \sum_{i=1}^\infty z_i \beta_i$$
with the $\beta_i$'s distinct elements of $[\alpha]$ and $z_1 \in \C \setminus \{ 0 \}$.  Then
$$ \pi_E (s_{\beta_1}^*) \vec{k} = \sum_{i=1}^\infty z_i \pi_E (s_{\beta_1}^*) \beta_i = z_1 v_0 + \sum_{i \in A} z_i \gamma_i$$
where
$$A := \{ i \in \N :  \beta_i = \beta_1 \gamma_i  \text{ for some $\gamma_i \in E^* \setminus E^0$} \}.$$  Note that $A \subseteq \{ 2, 3, \ldots \}$ and for each $i \in A$, the path $\gamma_i$ is a cycle (necessarily of length one or greater) with $s(\gamma_i) = r(\gamma_i) = v_0$.
If $A = \emptyset$, then $z_1 v_0 = \pi_E (s_{\beta_1}^*) \vec{k} \in \mathcal{K}$, and $v_0 \in \mathcal{K}$, so that $\mathcal{K} = \Hi_{[\alpha] }$ by Lemma~\ref{inv-subspace-with-basis-elt-all-lem}(b), and $\pi_E |_{ \Hi_{ [\alpha] } }$ is an irreducible algebra representation.  (Note that if $v_0$ is a sink, then no path can extend $\beta_1$ because $r(\beta_1) = v_0$.  Hence we always have $A = \emptyset$ whenever $v_0$ is a sink.)

If $A$ is nonempty, define 
$$\vec{h} :=  \pi_E (s_{\beta_1}^*) \vec{k} \in \mathcal{K}$$
and by relabeling the indices in $A$, write
$$\vec{h} = z_1 v_0 + \sum_{i=1}^\infty y_i \gamma_i \in \mathcal{K}$$
with $y_i \in \C$ for all $i \in \N$ and the $\gamma_i$'s distinct cycles based at $v_0$.

Let $\epsilon > 0$.  Choose $n \in \N$ such that 
$$ \left\| \sum_{i=n+1}^\infty y_i \gamma_i \right\| < \epsilon.$$
Consider the first edge in each of the cycles $\gamma_1, \ldots, \gamma_n$, and define
$$ B := \{ e \in E^1 : \text{$e$ is equal to the first edge of $\gamma_i$ for some $1 \leq i \leq n$} \}.$$
Note that $B$ is a finite subset of $E^1$, and $\pi_E (\sum_{e \in B} s_es_e^*)v_0 = 0$ .  Furthermore, 
$$ \pi_E \left( \sum_{e \in B} s_es_e^* \right) \vec{h} = \sum_{i=1}^n y_i \gamma_i + \sum_{i \in C} y_i \gamma_i$$
where $$C := \{ i \in \N : i \geq n+1 \text{ and the first edge of $\gamma_i$ is an element of $B$} \}.$$  Note that $C \subseteq \{ n+1, n+2, \ldots \}$.  Define $D:=  \{ n+1, n+2, \ldots \} \setminus C$.  Then
$$ \vec{h} - \pi_E \left( \sum_{e \in B} s_es_e^* \right) \vec{h}
= z_1 v_0 + \sum_{i \in D} y_i \gamma_i$$
and using the fact the $\gamma_i$'s are orthonormal
\begin{align*}
\left\| \left(  \vec{h} - \pi_E \left( \sum_{e \in B} s_es_e^* \right) \vec{h} \right) - z_1 v_0  \right\| &= 
\left\| \sum_{i \in D} y_i \gamma_i \right\| = \left( \sum_{i \in D} \| y_i \|^2 \right)^{1/2}
\\
&\leq \left( \sum_{i =n+1}^\infty \| y_i \|^2 \right)^{1/2} = \left\| \sum_{i=n+1}^\infty y_i \gamma_i \right\| < \epsilon.
\end{align*}
Since $\mathcal{K}$ is $\pi_E$-invariant,  $\vec{h} - \pi_E \left( \sum_{e \in B} s_es_e^* \right) \vec{h} \in \mathcal{K}$.  Because $\epsilon > 0$ was arbitrary, $z_1 v_0$ is in the closure of $\mathcal{K}$, and since $\mathcal{K}$ is closed, we conclude $z_1 v_0 \in \mathcal{K}$.  In addition, since $z_1 \neq 0$ and $\mathcal{K}$ is a subspace, we have $v_0 \in \mathcal{K}$.  By Lemma~\ref{inv-subspace-with-basis-elt-all-lem}(b) it follows that $\mathcal{K} = \Hi_{[\alpha] }$.  Hence $\pi_E |_{ \Hi_{ [\alpha] }}$ is an irreducible $*$-representation.
\end{proof}

\begin{notation} \label{rep-for-st-class-not}
Let $E = (E^0, E^1, r, s)$ be a graph.  By Proposition~\ref{restriction-irreducible-prop}(a), if $k$ is a field and $\alpha \in \partial E$, then the subspace $V_{ ([\alpha], k)} := \operatorname{span}_k [\alpha]$ is $\rho_{E, k}$-invariant, and the restriction $\rho_{E, k} |_{ V_{ ([\alpha], k)} } $ is an irreducible algebra representation.  For convenience of notation, whenever $\alpha \in \partial E$ we let
$$\rho_{ ([\alpha], k)} := \rho_{E, k} |_{ V_{ ([\alpha], k)} } $$
denote this restriction.

Likewise, by Proposition~\ref{restriction-irreducible-prop}(b) for any $\alpha \in \partial E$ the subspace $\Hi_{ [\alpha] } := \ell^2 ( [ \alpha ])$ is $\pi_E$-invariant, and the restriction $\pi_E |_{ H_{ [\alpha] } }$ is an irreducible $*$-representation.  For convenience of notation, whenever $\alpha \in \partial E$ we let
$$\pi_{ [\alpha] } := \pi_E |_{ H_{ [\alpha] } }$$
denote this restriction.
\end{notation}

\begin{proposition} \label{equiv-repn-implies-st-equiv-paths-prop}
Let $E = (E^0, E^1, r, s)$ be a graph.
\begin{itemize}
\item[(a)] Let $k$ be a field.  If $\alpha, \beta \in \partial E$ and there exists a nonzero $k$-linear map $T :  V_{ ([\alpha], k)} \to  V_{ ([\beta], k)}$ such that
$$ T \circ \rho_{( [\alpha], k)}(a) = \rho_{( [\beta], k)} (a) \circ T  \qquad \text{for all $a \in L_k(E)$},$$
then $[\alpha ] = [\beta ]$ and $T = z \operatorname{Id}_{ V_{ ([\alpha], k ) } }$.  Consequently, for any $\alpha, \beta \in \partial E$ with $[ \alpha ] \neq [ \beta ]$ we have
$$ \operatorname{End}_{L_k(E)} V_{ ([\alpha], k ) }  \cong k \quad \text{ and } \quad 
\operatorname{Hom}_{L_k(E)} (V_{ ([\alpha], k ) }, V_{ ([\beta], k ) }) = \{ 0 \}.
$$
\item[(b)]  If $\alpha, \beta \in \partial E$ and there exists a nonzero bounded linear map $T :  \Hi_{ [\alpha] } \to  \Hi_{ [\beta ]}$ such that
$$ T \circ \pi_{ [\alpha] } (a)  =\pi_{ [\beta] } (a) \circ T  \qquad \text{for all $a \in C^*(E)$},$$
then $[\alpha ] = [\beta ]$ and $T = z \operatorname{Id}_{ \Hi_{ [\alpha] } }$.  Consequently, for any $\alpha, \beta \in \partial E$ with $[ \alpha ] \neq [ \beta ]$ we have
$$ \operatorname{End}_{C^*(E)} \Hi_{ [\alpha] }  \cong \C, \qquad  \qquad \operatorname{Unitary}_{C^*(E)} \Hi_{ [\alpha] }  \cong \T,$$
and
$$
\operatorname{Hom}_{C^*(E)} (\Hi_{ [\alpha ] }, \Hi_{ [\beta] }) = \{ 0 \}.
$$
\end{itemize}
\end{proposition}

\begin{proof}
We shall prove (b) first.  

Suppose $T :  \Hi_{ [\alpha] } \to  \Hi_{ [\beta] }$ is a nonzero bounded linear map such that
$$ T \circ \pi_{ [\alpha] } (a)  =\pi_{ [\beta] } (a) \circ T  \qquad \text{for all $a \in C^*(E)$}.$$
Since $T$ is nonzero and $\Hi_{ [\alpha] } := \ell^2( [\alpha] ) =\overline{\operatorname{span}} \ [\alpha]$, there exists $\gamma \in [\alpha]$ such that $T (\gamma) \neq 0$.  Since $\Hi_{ [\beta] } := \ell^2( [\beta] ) =\overline{\operatorname{span}} \ [\beta]$ we may write
$$ T(\gamma) = \sum_{i \in I} z_i \beta_i$$
for $z_i \in \C \setminus \{ 0 \}$, $\beta_i \in [\beta]$, and with the $\beta_i$'s distinct for all $i \in I$.  Note that since $T (\gamma)$ is nonzero, there is at least one (nonzero) term in this sum.  

For each $N \in \N$ with $N \leq | \gamma |$, define 
$$\mu_N := e_1 \ldots e_N \quad \text{ with each $e_i \in E^1$}$$
to be the finite path consisting of the first $N$ edges of $\gamma$.  Then
\begin{align}
\sum_{i \in I} z_i  \pi_{ [\alpha]} (s_{\mu_N} s_{\mu_N}^*) \beta_i &= \pi_{ [\alpha] } (s_{\mu_N} s_{\mu_N}^*) \left( \sum_{i \in I}  \beta_i \right) =  \pi_{ [\alpha] } (s_{\mu_N} s_{\mu_N}^*) \circ T (\gamma) \notag \\
&= T \circ \pi_{ [\alpha]} (s_{\mu_N} s_{\mu_N}^*) (\gamma) = T (\gamma) = \sum_{i \in I}  z_i \beta_i. \label{lin-ind-sum1-eq}
\end{align}
For each $i \in I$ the element  $\pi_{ [\alpha] } (s_{\mu_N} s_{\mu_N}^*) \beta_i$ is either $\beta_i$ or zero depending on whether $\mu_N$ is an initial path of $\beta_i$ or not.  However, since $\{ \beta_i \}_{i \in I}$ is linearly independent and $z_i \neq 0$ for all $i \in I$, the only way we can have the equality of the two linear combinations in \eqref{lin-ind-sum1-eq} is if $\pi_{ [\alpha] } (s_{\mu_N} s_{\mu_N}^*) \beta_i = \beta_i$ for all $i \in I$.  Thus for each $i \in I$ we have that $\mu_N := e_1 \ldots e_N$ is an initial path of $\beta_i$.  Because $N \in \N$ is an arbitrary natural number less than the length of $\gamma$, it follows that $\beta_i$  extends $\gamma$ for each $i \in I$.

In a similar manner, fix $j \in I$.  For any $N \in \N$ with $N \leq | \beta_j|$, let
$$ \nu_N :=  \beta_{j,1} \beta_{j,2} \ldots \beta_{j,N} \qquad \text{ with each $\mu_{j,k} \in E^1$}$$
be the finite path consisting of the first $N$ edges of $\beta_j$.  Then
\begin{align}
T \circ \pi_{ [\alpha]} (s_{\nu_N} s_{\nu_N}^*) (\gamma) &= \pi_{ [\beta] } (s_{\nu_N} s_{\nu_N}^*) \circ T (\gamma) =  \pi_{[\beta]} (s_{\nu_N} s_{\nu_N}^*) \left( \sum_{i \in I} z_i \beta_i \right) \notag \\
&=  \sum_{i \in I} z_i \pi_{ [\beta] } ( s_{\nu_N} s_{\nu_N}^* ) \beta_i \label{lin-ind-sum1-eq}
\end{align}
We observe that $\pi_{ [\alpha] } ( s_{\nu_N} s_{\nu_N}^* ) \beta_i$ is equal to either $\beta_i$ or zero, depending on whether $\nu_N$ is an initial path of $\beta_i$ or not.  Since $\{ \beta_i \}_{i \in I}$ is linearly independent, the nonzero terms in the sum $\sum_{i \in I} z_i \pi_{ [\beta] } ( s_{\nu_N} s_{\nu_N}^* ) \alpha_i$ are linearly independent.  In addition, since $z_j \neq 0$ and $\pi_{ [\alpha] } ( s_{\nu_N} s_{\nu_N}^* ) \beta_j = \beta_j \neq 0$, we conclude that 
$$\sum_{i \in I}  z_i \pi_{ [\beta] } ( s_{\nu_N} s_{\nu_N}^* ) \beta_i \neq 0.$$  Because $T$ is linear, and thus $T(0) = 0$, the fact $T \circ \pi_{ [\alpha] } (s_{\nu_N} s_{\nu_N}^*) (\beta) \neq 0$ implies $\pi_{ [\alpha] } (s_{\nu_N} s_{\nu_N}^*) (\gamma) \neq 0$.  Hence $\pi_{ [\alpha] } (s_{\nu_N} s_{\nu_N}^*) (\gamma) = \gamma$, and $\gamma$ extends $\nu_N$.  Because $N \in \N$ is an arbitrary natural number with $N \leq | \beta_j|$, we conclude $\gamma$ extends $\beta_j$ for each $j \in I$.  Together with the previous paragraph, this implies $\beta_i = \gamma$ for all $i \in I$.

Since the $\beta_i$ were chosen distinct and $\beta_i = \gamma$ for all $i \in I$, we conclude that there is a single $\beta_i$.  In other words, $I = \{ i \}$ is a singleton set with $\beta_i = \gamma$, and hence $T( \gamma) =  z \gamma$ for some nonzero $z \in \C$.  In particular, since $\gamma \in [\alpha]$ and $\beta_i \in [\beta]$, we have $\alpha \sim_{st} \gamma = \beta_i \sim_{st} \beta$.  Thus $[\alpha] = [\beta ]$.

Finally, for any $\delta \in [\alpha]$ we have $\delta \sim_{st} \alpha \sim_{st} \gamma$, and hence there exists $\mu, \nu \in E^*$ and $\lambda \in \partial E$ such that $\delta = \mu \lambda$ and $\gamma = \nu \lambda$.  Thus
\begin{align*}
T (\delta) &= T \circ \pi_{ [\alpha ] } (s_\mu s_\nu^*) \gamma =  \pi_{ [\beta ] } (s_\mu s_\nu^*) \circ T (\gamma ) 
\\
&=  \pi_{ [\beta ] } (s_\mu s_\nu^*) (z\gamma) = z  \pi_{ [\beta ] } (s_\mu s_\nu^*) (\gamma) = z \delta.
\end{align*}
Since $T (\delta) = z \delta$ for all $\delta \in [\alpha]$ and $\Hi_{ [\alpha] } := \ell^2( [\alpha] ) =\overline{\operatorname{span}} \ [\alpha]$, we conclude  $T = z \operatorname{Id}_{ \Hi_{ [\alpha] } }$. Finally, we observe that $T = z \operatorname{Id}_{ \Hi_{ [\alpha] } }$ is a unitary if and only if $|z| = 1$ if and only if $z \in \T$.

The proof of (a) is very similar.
\end{proof}

We are now in a position to decompose the representations of Proposition~\ref{representation-exists-prop} as direct sums of irreducible representations.  Moreover, we will show these irreducible subrepresentations are pairwise inequivalent.

\begin{theorem} \label{decomposition-thm}
Let $E = (E^0, E^1, r, s)$ be a graph.  
\begin{itemize}
\item[(a)] For any field $k$ , let $\rho_{E,k} : L_K(E) \to  \End_k V_{ ( \partial E ,k) }$ denote the representation of Proposition~\ref{representation-exists-prop}(a).  For each $\alpha \in \partial E$ the subspace $V_{ ([\alpha],k) } := \operatorname{span}_k [\alpha]$ is invariant for $\rho_{E,k}$, and the restriction $\rho_{ ( \alpha, k) } :=  \rho_{E,k} |_{ V_{ ([\alpha],k) } }$ is an irreducible representation of $L_K(E)$ on $V_{ ([\alpha],k) }$.  In addition, $$V_{(\partial E,k)} = \bigoplus_{ [\alpha] \in \widetilde{\partial} E }  V_{ ([ \alpha ], k) }\qquad  \text { and } \qquad  \rho_{E,k} = \bigoplus_{  [\alpha] \in \widetilde{\partial} E } \rho_{ ( \alpha, k) }$$
is a decomposition of  $\rho_{E,k}$ as a direct sum of irreducible representations.

\noindent Moreover, for any $\alpha, \beta \in \partial E$ the following are equivalent:
\begin{itemize}
\item[(i)] $\rho_{( [\alpha], k)}$ is algebraically equivalent to $\rho_{( [\beta], k)}$.
\item[(ii)] $\rho_{( [\alpha], k)} = \rho_{( [\beta], k)}$.
\item[(iii)] $\alpha \sim_{st} \beta$ (or equivalently, $[\alpha ] = [ \beta]$).
\end{itemize}

\medskip

\item[(b)] Let $\pi_E : C^*(E) \to \B ( H_{ \partial E})$ denote the representation of Proposition~\ref{representation-exists-prop}(b).  For each $\alpha \in \partial E$ the closed subspace $\Hi_{ [\alpha] } := \ell^2( [\alpha ] )$ is invariant for $\pi_E$, and the restriction $\pi_{ [ \alpha ] } :=  \pi_E |_{ \Hi_{ [\alpha] } }$ is an irreducible representation of $C^*(E)$ on $\Hi_{ [\alpha] }$.   In addition, $$\Hi_E = \bigoplus_{ [\alpha] \in \widetilde{\partial} E } \Hi_{ [\alpha ] } \qquad \text{ and } \qquad \pi_E = \bigoplus_{  [\alpha] \in \widetilde{\partial} E } \pi_{ [\alpha ] }$$
is a decomposition of  $\pi_E$ as a direct sum of irreducible representations.

\noindent Moreover, for any $\alpha, \beta \in \partial E$ the following are equivalent:
\begin{itemize}
\item[(i)] $\pi_{ [\alpha] }$ is unitarily equivalent to $\pi_{ [\beta]}$.
\item[(ii)] $\pi_{ [\alpha] } = \pi_{ [\beta]}$.
\item[(iii)] $\alpha \sim_{st} \beta$ (or equivalently, $[\alpha ] = [ \beta]$).
\end{itemize}

\end{itemize}
\end{theorem}

\begin{proof}
For (a), it follows from Proposition~\ref{restriction-irreducible-prop}(a) that for each $\alpha \in \partial E$ the subspace $V_{ ([\alpha],k) } := \operatorname{span}_k [\alpha]$ is invariant for $\rho_{E,k}$, and the restriction $\rho_{ ( \alpha, k) } :=  \rho_{E,k} |_{ V_{ ([\alpha],k) } }$ is an irreducible representation.  Since shift-tail equivalence is an equivalence relation, the shift-tail equivalence classes partition the basis $\partial E$ of 
$V_{(\partial E,k)}$.  Hence $V_{(\partial E,k)} = \bigoplus_{ [\alpha] \in \widetilde{\partial} E } V_{ ([ \alpha ], k) }$ and $\rho_{E,k} = \bigoplus_{  [\alpha] \in \widetilde{\partial} E } \rho_{ ( \alpha, k) }$ is a decomposition of $\rho_{E,k}$ as a direct sum of irreducible representations.  For the equivalence of (i)--(iii) in part~(a), observe that the implications (iii) $\implies$ (ii) $\implies$ (i) are trivial.  In addition, the implication (i) $\implies$ (iii) follows from Proposition~\ref{equiv-repn-implies-st-equiv-paths-prop}(a).

For (b), it follows from Proposition~\ref{restriction-irreducible-prop}(b) that for each $\alpha \in \partial E$ the closed subspace $H_{ [\alpha] } := \ell^2( [\alpha ] )$ is invariant for $\pi_E$, and the restriction $\pi_{ [ \alpha ] } :=  \pi_E |_{ H_{ [\alpha] } }$ is an irreducible representation. Since shift-tail equivalence is an equivalence relation, the shift-tail equivalence classes partition the basis $\partial E$ of $H_E$.  Hence $H_E = \bigoplus_{ [\alpha] \in \widetilde{\partial } E } H_{ [\alpha ] }$ and $\pi_E = \bigoplus_{  [\alpha] \in \widetilde{\partial } E } \pi_{ [\alpha ] }$ is a decomposition of $\pi_E$ as a direct sum of irreducible representations.  For the equivalence of (i)--(iii) in part~(b), observe that the implications (iii) $\implies$ (ii) $\implies$ (i) are trivial.  In addition, the implication (i) $\implies$ (iii) follows from Proposition~\ref{equiv-repn-implies-st-equiv-paths-prop}(b).
\end{proof}

\begin{corollary} \label{number-STE-algebraic-unitary-equiv-classes-cor}
Let $E$ be graph.
\begin{itemize}
\item[(1)]  For the graph $C^*$-algebra $C^*(E)$, we have

$$
\left\vert \left\{ 
\begin{matrix}
 \text{shift-tail} \\
 \text{equivalence} \\
 \text{classes of $\partial E$ }
\end{matrix}
\right\} \right\vert
\leq
\left\vert \left\{ 
\begin{matrix}
 \text{unitary equivalence} \\
 \text{classes of irreducible} \\
 \text{$*$-repns.~of $C^*(E)$ }
\end{matrix}
\right\} \right\vert.
$$

$ $

\noindent Consequently, if $C^*(E)$ has the property that any two $*$-representations of $C^*(E)$ are unitarily equivalent, then any two paths in $\partial E$ are shift-tail equivalent. 

\item[(2)]  If $k$ is any field, then for the Leavitt path algebra $L_k(E)$ we have

$$
\left\vert \left\{ 
\begin{matrix}
 \text{shift-tail} \\
 \text{equivalence} \\
 \text{classes of $\partial E$ }
\end{matrix}
\right\} \right\vert
\leq
\left\vert \left\{ 
\begin{matrix}
 \text{algebraic equivalence} \\
 \text{classes of irreducible} \\
 \text{repns.~of $L_k(E)$ }
\end{matrix}
\right\} \right\vert.
$$

$ $

\noindent Consequently, if there exists a field $k$ such that any two representations of $L_k(E)$ are algebraically equivalent, then any two paths in $\partial E$ are shift-tail equivalent.
\end{itemize}
\end{corollary}

\section{Lemmas Relating Graph Structure and Representation Theory} \label{lemmas-relating-graphs-repns-sec}

\begin{lemma}[Cf.~Proposition~4.1 of \cite{AR}] \label{two-distinct-cycles-implies-uncountable}
Let $E$ be a graph.  If there is a vertex of $E$ that is the base point of two distinct simple cycles, then there are an uncountable number of shift-tail equivalence classes of $\partial E$.
\end{lemma}

\begin{proof}
Let $\mu$ and $\nu$ denote distinct simple cycles based at $v$.  Each sequence of $\mu$'s and $\nu$'s determines an infinite path in $\partial E$.  There are uncountably many such sequences, and hence uncountably many such infinite paths.  Because $\mu$ and $\nu$ are distinct, any two sequences of $\mu$'s and $\nu$'s are shift-tail equivalent if and only if they are equal past a certain point.  Consequently any path equal to a sequence of $\mu$'s and $\nu$'s is shift-tail equivalent to only a countable number of paths that are sequences of $\mu$'s and $\nu$'s.  Since the countable union of countable sets is countable, we conclude that there must be an uncountable number of shift-tail equivalence classes of $\partial E$.
\end{proof}

\begin{lemma} \label{C-star-countable-no-cycles-lem}
Let $E$ be a graph.  If the collection of unitary equivalence classes of irreducible $*$-representations of $C^*(E)$ is countable, then $E$ has no cycles.
\end{lemma}

\begin{proof}
We shall suppose that $E$ has a cycle $\mu$ and obtain a contradiction.  Without loss of generality, we may assume $\mu$ is a simple cycle.   Consider two cases.

\medskip

\noindent \textsc{Case I:}  $\mu$ is not the only simple cycle based at $s(\mu)$.

\noindent Then there exists a second simple cycle $\nu$ based at $s(\mu)$.  By Lemma~\ref{two-distinct-cycles-implies-uncountable}, there are an uncountable number of shift-tail equivalence classes of $\partial E$.  Corollary~\ref{number-STE-algebraic-unitary-equiv-classes-cor}(1) then implies that $C^*(E)$ has an uncountable number of algebraic equivalence classes of irreducible $*$-representations, which is a contradiction.

\medskip

\noindent \textsc{Case II:} $\mu$ is the unique simple cycle based at $s(\mu)$.

\noindent Let 
$$H := \{ v \in E^0 : \text{there is no path from $v$ to $s(\mu)$} \}$$
and observe that $H$ is saturated hereditary.  Let $\mathcal{I}_{(H, B_H)}$ be the closed two-sided ideal in $C^*(E)$ of Definition~\ref{ideals-def}.  Then $C^*(E) / \mathcal{I}_{(H, B_H)} \cong C^*(E \setminus (H, B_H) )$, where $E \setminus (H, B_H)$ is the graph described in Remark~\ref{ideals-quotients-rem}.  

Then $\mu$ is a cycle in $E \setminus (H, B_H)$.  Since $\mu$ is the unique simple cycle based at $s(\mu)$ in $E$, and since there is a  path from every vertex in $E \setminus (H, B_H)$ to $s(\mu)$, we conclude $\mu$ is a cycle with no exits in $E \setminus (H, B_H)$.  Let $J$ be the ideal of $C^*(E\setminus (H,B_H))$ generated by the projections corresponding to vertices on $\mu$.  Then $J$ is Morita equivalent to $C ( \mathbb{T} )$, and since $C ( \mathbb{T} )$ has uncountably many irreducible $*$-representations with pairwise distinct kernels, we conclude there exists an uncountable collection $\{ \pi_i : i \in I \}$ of irreducible $*$-representations of $J$ with pairwise distinct kernels.  Moreover, by \cite[Theorem~5.5.1]{Mur} each of these irreducible $*$-representations may be extended to obtain an uncountable collection $\{ \overline{\pi}_i : i \in I \}$ of irreducible $*$-representations of $C^*(E \setminus (H,B_H))$ with pairwise distinct kernels.  If we let $q : C^*(E) \to C^*(E) / I_{(H,B_H)} \cong C^*(E\setminus (H,B_H))$ denote the quotient map, then $\{ \overline{\pi}_i \circ q : i \in I \}$ is an uncountable collection of irreducible $*$-representations of $C^*(E)$ with pairwise distinct kernels.  Since the kernels of these $*$-representations are pairwise distinct, the $*$-representations are pairwise inequivalent.  Hence $C^*(E)$ has an uncountable number of algebraic equivalence classes of irreducible $*$-representations, which contradicts our hypotheses.

Thus we have obtained a contradiction in both cases, which implies $E$ has no cycles.
\end{proof}

\begin{remark}
When $k$ is a countable field the algebraic equivalence classes of irreducible representations of $k[x,x^{-1}]$ is countable.  (There is one irreducible representation for each irreducible polynomial.)
Thus the results of Lemma~\ref{C-star-countable-no-cycles-lem} do not hold for Leavitt path algebras; i.e., there are examples of graphs with cycles whose Leavitt path algebra has only countably many equivalence classes of irreducible representations.  Despite this, in the next lemma we are able to show that if a Leavitt path algebra has a unique irreducible representation (up to equivalence), then the graph has no cycles.
\end{remark}

\begin{lemma} \label{all-ste-implies-sat-hered-all-lem}
If $E$ is a graph and any two boundary paths of $E$ are shift-tail equivalent, then any nonempty saturated hereditary subset of $E$ is equal to all of $E^0$.
\end{lemma}

\begin{proof}
Let $H$ be a nonempty saturated hereditary subset of $E$.  Choose $v \in H$.  Since $H$ is hereditary, either $v$ is a singular vertex or there exists an edge $e_1$ with $s(e_1) = v$ and $r(e_1) \in H$.  Likewise, either $r(e_1)$ is a singular vertex or there exists an edge $e_2$ with $s(e_2) = r(e_1)$ and $r(e_2) \in H$.  Continuing in the manner, we produce a boundary path (i.e., either an infinite path or a finite path ending at a singular vertex) with all its vertices in $H$.

For the sake of contradiction, suppose $H \neq E^0$.  Then there exists $w \in E^0 \setminus H$.  Since $H$ is saturated and $w \notin H$, either $w$ is a singular vertex or there exists an edge $f_1$ with $s(f_1) = w$ and $r(f_1) \notin H$.  Likewise, either $r(f_1)$ is a singular vertex or there exists an edge $f_2$ with $s(f_2) = r(f_1)$ and $r(f_2) \notin H$.  Continuing in the manner, we produce a boundary path (i.e., either an infinite path or a finite path ending at a singular vertex) with all its vertices in $E^0 \setminus H$.  Since all vertices of $e_1 e_2 \ldots$ are in $H$ and all vertices of $f_1 f_2 \ldots$ are not in $H$, we have two boundary paths of $E$ that are not shift-tail equivalent, which is a contradiction.  Hence we conclude $H = E^0$.
\end{proof}

\begin{lemma} \label{LPA-countable-no-cycles-lem}
Let $E$ be a graph and let $k$ be a field.  If $L_k(E)$ has the property that any two irreducible representations of $L_k(E)$ are algebraically equivalent, then $E$ has no cycles.  
\end{lemma}

\begin{proof}
We shall suppose that $E$ has a cycle $\mu$ and obtain a contradiction.  Without loss of generality, we may assume $\mu$ is a simple cycle.  

Since any two irreducible representations of $L_k(E)$ are algebraically equivalent, it follows from Corollary~\ref{number-STE-algebraic-unitary-equiv-classes-cor}(2) that any two boundary paths in $\partial E$ are shift-tail equivalent.  It then follows from Lemma~\ref{two-distinct-cycles-implies-uncountable} that $\mu$ is the only simple cycle based at $s(\mu)$.  Consider two cases.

\noindent \textsc{Case I:} $\mu$ is the unique simple cycle based at $s(\mu)$ and $\mu$ has an exit.

Let $f \in E^1$ be an exit for $\mu$.  Since $\mu$ is the unique cycle based at $s(\mu)$, we conclude that there is no path from $r(f)$ to $s(\mu)$.  Starting at $r(f)$ we may find $\alpha \in \partial E$ with $s(\alpha) = r(f)$.  (Note that $\alpha$ is either a finite path starting at $r(f)$ and ending at a singular vertex or an infinite path starting at $r(f)$.)  Since there is no finite path from $r(f)$ to $s(\mu)$, it follows that the boundary paths $\alpha$ and $\mu \mu \mu \ldots$ are not shift-tail equivalent, which is a contradiction.

\noindent \textsc{Case II:} $\mu$ is the unique simple cycle based at $s(\mu)$ and $\mu$ does not have an exit.

Let $H := \{ w \in E^0 : \text{$w$ is a vertex on $\mu$} \}$, and note that since $\mu$ does not have an exit, $H$ is hereditary.  Corollary~\ref{number-STE-algebraic-unitary-equiv-classes-cor}(2) implies that any two boundary paths of $E$ are shift-tail equivalent, and hence Lemma~\ref{all-ste-implies-sat-hered-all-lem} implies any nonempty saturated hereditary subset of $E$ is equal to all of $E^0$.  Thus $I_H = I_{\overline{H}} = I_{E^0} = L_k(E)$.  Since $\mu$ is a cycle with no exits, $I_H$ is Morita equivalent to the Leavitt path algebra of a single cycle, which in turn is Morita equivalent to the Laurent algebra $k[x,x^{-1}]$.  Hence $L_k(E)$ is Morita equivalent to $k[x,x^{-1}]$, and since $k[x,x^{-1}]$ has an infinite number of irreducible representations that are pairwise algebraically inequivalent, so does $L_k(E)$, which is a contradiction.

Thus we have a contradiction in both cases, and we conclude $E$ has no cycles.
\end{proof}

\begin{lemma} \label{line-points-exist-first-lem}
Let $E = (E^0, E^1, r, s)$ be a graph with no cycles and no line points.  Then for any vertex $v \in E^0$ and any infinite path $\alpha := f_1 f_2 \ldots \in E^\infty$ there exists a finite path $\mu := e_1 \ldots e_n$ such that $s(e_1) = v$ and $e_n \neq f_i$ for all $i \in \N$.
\end{lemma}

\begin{proof}
Since $E$ has no line points, $E$ has no sinks.  Thus there exists $e_1 \in E^1$ with $s(e_1) = v$.  If $e_1$ is not an edge on $\alpha$, we may let $\mu := e_1$.  If instead $e_1$ is an edge on $\alpha$, then since $E$ contains no cycles and no line points we may extend the edge $e_1$ to a finite subpath $e_1 \ldots e_{n-1}$ of $\alpha$ with $r(e_{n-1})$ a bifurcating vertex.  Since $r(e_{n-1})$ is on the path $\alpha$, we have $r(e_{n-1}) = s(f_m)$ for some $m \in \N$.  Also, since $r(e_{n-1})$ is a bifurcating vertex there exists an edge $e_n$ with $s(e_n) = r(e_{n-1})$ and $e_n \neq f_m$.  Let $\mu := e_1 \ldots e_{n-1} e_n$.  Because $E$ has no cycles, the fact that
$s(e_n) = s(f_m)$ and $e_n \neq f_m$ implies that we must have $e_n \neq f_i$ for all $i \in \N$.
\end{proof}

\begin{lemma} \label{ste-implies-line-point-lem}
Let $E = (E^0, E^1, r, s)$ be a graph.  If $E$ has no cycles and $E$ has a countable number of shift-tail equivalence classes of boundary paths, then there exists a line point in $E$.
\end{lemma}

\begin{proof}
We shall prove the result by contradiction.  Suppose $E$ has no cycles, $E$ has a countable number of shift-tail equivalence classes of boundary paths, and $E$ contains no line points.  Since $E$ has no line points, $E$ has no sinks and  every boundary path of $E$ is an infinite path. For each shift-tail equivalence class choose a representative infinite path from that equivalence class, and list these representatives as
\begin{align*}
\alpha_1 &:= e_{1,1} e_{1,2} e_{1,3} \ldots \\
\alpha_2 &:= e_{2,1} e_{2,2} e_{2,3} \ldots \\
\alpha_3 &:= e_{3,1} e_{3,2} e_{3,3} \ldots \\
&\quad \vdots
\end{align*}
Note that this list is countable (i.e., either finite or countably infinite), and every infinite path is shift-tail equivalent to one of the infinite paths in this list.

Construct an infinite path inductively as follows.  First, choose any vertex $v \in E^0$, and use Lemma~\ref{line-points-exist-first-lem} to choose a finite path $\mu_{1,1}$ such that $s(\mu_{1,1}) = v$ and the final edge of $\mu_{1,1}$ is not an edge on $\alpha_1$.

Second, use Lemma~\ref{line-points-exist-first-lem} to choose a finite path $\mu_{2,1}$ such that $s(\mu_{2,1}) = r(\mu_{1,1})$ and the final edge of $\mu_{2,1}$ is not an edge on $\alpha_1$.  Then use Lemma~\ref{line-points-exist-first-lem} to choose a finite path $\mu_{2,2}$ such that $s(\mu_{2,2}) = r(\mu_{2,1})$ and the final edge of $\mu_{2,2}$ is not an edge on $\alpha_2$.

Third, use Lemma~\ref{line-points-exist-first-lem} to choose a finite path $\mu_{3,1}$ such that $s(\mu_{3,1}) = r(\mu_{2,2})$ and the final edge of $\mu_{3,1}$ is not an edge on $\alpha_1$.  Then use Lemma~\ref{line-points-exist-first-lem} to choose a finite path $\mu_{3,2}$ such that $s(\mu_{3,2}) = r(\mu_{3,1})$ and the final edge of $\mu_{3,2}$ is not an edge on $\alpha_2$.  Then use Lemma~\ref{line-points-exist-first-lem} to choose a finite path $\mu_{3,3}$ such that $s(\mu_{3,3}) = r(\mu_{3,2})$ and the final edge of $\mu_{3,3}$ is not an edge on $\alpha_3$.

Continuing in this manner we produce an infinite path
$$ \beta := \mu_{1,1} \, \mu_{2,1}  \, \mu_{2,2}  \, \mu_{3,1}  \, \mu_{3,2}  \, \mu_{3,3}  \, \mu_{4,1}  \, \mu_{4,2}  \, \mu_{4,3}  \, \mu_{4,4} \ldots$$
with the property that for each $i,j \in \N$ the final edge of $\mu_{i,j}$ is not an edge on $\alpha_j$.

Consequently, for each $\alpha_j$ in our above list, the infinite path $\beta$ contains infinitely many edges not on $\alpha_j$ (namely, the final edge of each $\mu_{i,j}$ for each $i \in \N$).  Consequently, $\beta$ is not shift equivalent to any of the $\alpha_j$ in our above list, which is a contradiction.
\end{proof}

\section{Naimark's Problem for Graph Algebras} \label{Naimark-Problem-sec}

We now state and prove our main result regarding Naimark's Problem.

\begin{theorem} \label{Naimark-thm}
Let $E = (E^0, E^1, r, s)$ be a (not necessarily countable) graph.  Then the following are equivalent.
\begin{itemize}
\item[(1)] For some field $k$, any two irreducible representations of $L_k(E)$ are algebraically equivalent.  (Equivalently: For some field $k$, any two simple left/right $L_k(E)$-modules are isomorphic.)
\item[(2)] For every field $k$, any two irreducible representations of $L_k(E)$ are algebraically equivalent. (Equivalently: For every field $k$, any two simple left/right $L_k(E)$-modules are isomorphic.)
\item[(3)] Any two irreducible $*$-representations of $C^*(E)$ are unitarily equivalent.
\item[(4)] The graph $E$ has no cycles, and any two boundary paths of $E$ are shift-tail equivalent. 
\item[(5)] There exists a line point $v \in E^0$ such that $E^0 = \overline{T(v)}$ (i.e., $E^0$ is the saturation of the hereditary subset $T(v)$).
\item[(6)] There exists an index set $\Lambda$ and there exists a field $k$ such that $L_k(E) \cong M_\Lambda (k)$ (as $*$-algebras).
\item[(7)] There exists an index set $\Lambda$ such that $L_k(E) \cong M_\Lambda (k)$ (as $*$-algebras) for every field $k$.
\item[(8)] There exists a Hilbert space $\mathcal{H}$ such that $C^*(E) \cong \mathbb{K}(\mathcal{H})$.
\end{itemize}
\end{theorem}

\begin{proof}
We shall establish the following implications.
$$
\xymatrix{
(2) \ar@{=>}[r] & (1) \ar@{=>}[r] & (4) \ar@{=>}[d] &  (3) \ar@{=>}[l] \\
(6) \ar@{=>}[u] & (7) \ar@{=>}[l] & (5) \ar@{=>}[l] \ar@{=>}[r] & (8) \ar@{=>}[u]
}
$$

\noindent $(2) \implies (1)$: This is trivial.

\noindent $(1) \implies (4)$: Lemma~\ref{LPA-countable-no-cycles-lem} implies $E$ has no cycles, and Corollary~\ref{number-STE-algebraic-unitary-equiv-classes-cor}(2) implies any two boundary paths of $E$ are shift-tail equivalent.

\noindent $(4) \implies (5)$: Lemma~\ref{ste-implies-line-point-lem} implies there exists a line point $v \in E^0$.  Since $T(v)$ is hereditary, Lemma~\ref{all-ste-implies-sat-hered-all-lem} implies $\overline{ T(v) } = E^0$.

\noindent $(5) \implies (7)$:  Let $k$ be any field, and let $H:= T(v)$ so that $\overline{H} = E^0$ by hypothesis.  Since $H$ is hereditary, we have $I_H = I_{\overline{H}}$.  If we let $\Lambda := \{ e_1 \ldots e_n \in E^* : r(e_n) \in T(v) \text{ and } r(e_{n-1}) \notin T(v) \}$, then Lemma~\ref{ideal-of-line-point-matrix-and-compacts-lem} gives $L_k(E) = I_{E^0} = I_{\overline{H}} = I_H \cong M_\Lambda (k)$.

\noindent $(7) \implies (6)$: This is trivial.

\noindent $(6) \implies (2)$: Since $L_k(E) \cong M_\lambda (k)$ is Morita equivalent to $k$, and since any two irreducible representations of $k$ are algebraically equivalent, we conclude that any two irreducible representations of $L_k(E)$ are algebraically equivalent.

\noindent $(5) \implies (8)$: Let $H:= T(v)$ so that $\overline{H} = E^0$ by hypothesis.  Since $H$ is hereditary, we have $\mathcal{I}_H = \mathcal{I}_{\overline{H}}$.  If we let $\Lambda := \{ e_1 \ldots e_n \in E^* : r(e_n) \in T(v) \text{ and } r(e_{n-1}) \notin T(v) \}$, then Lemma~\ref{ideal-of-line-point-matrix-and-compacts-lem} gives $C^*(E) = \mathcal{I}_{E^0} = \mathcal{I}_{\overline{H}} = \mathcal{I}_H \cong \mathbb{K}(\ell^2(\Lambda))$.  Hence the claim holds for $\mathcal{H} := \ell^2 (\Lambda)$.

\noindent $(8) \implies (3)$: Since $C^*(E) \cong \mathbb{K}(\mathcal{H})$ is Morita equivalent to $\C$, and since any two irreducible representations of $\C$ are unitarily equivalent, we conclude that any two irreducible representations of $C^*(E)$ are unitarily equivalent.

\noindent $(3) \implies (4)$: Lemma~\ref{C-star-countable-no-cycles-lem} implies $E$ has no cycles, and Corollary~\ref{number-STE-algebraic-unitary-equiv-classes-cor}(1) implies any two boundary paths of $E$ are shift-tail equivalent.
\end{proof}

\begin{corollary}
Naimark's problem has a positive answer for the class of graph $C^*$-algebras.  In particular, if $A$ is a (not necessarily separable) graph $C^*$-algebra with the property that any two irreducible representations are unitarily equivalent, then $A \cong \K (\Hi)$ for some (not necessarily separable) Hilbert space $\Hi$.
\end{corollary}

\begin{corollary}
The algebraic analogue of Naimark's problem has a positive answer for the class of Leavitt path algebras.  In particular, if $A$ is the Leavitt path algebra of a (not necessarily countable) graph over a field $k$ with the property that any two irreducible representations of $A$ are algebraically equivalent, then $A \cong M_\Lambda (k)$ (as $*$-algebras) for some (possibly uncountable) index set $\Lambda$.
\end{corollary}

\begin{remark}
If a graph satisfies one (and hence all) of the equivalent conditions in the statement of Theorem~\ref{Naimark-thm}, that graph is downward directed and has no cycles, no infinite emitters, and at most one sink.  In addition, if the graph has a (unique) sink, then the graph has no infinite paths and every vertex in the graph can reach the sink.  If the graph has no sinks, then the graph has an infinite path containing a line point, and every vertex in the graph can reach a vertex on this infinite path.  
\end{remark}

\begin{remark}[Naimark's Problem for Associative Algebras]
It is well known that Naimark's Problem has an affirmative answer for any separable $C^*$-algebra \cite[Theorem~4]{Ros}.  In light of Theorem~\ref{Naimark-thm}, we see that the algebraic analogue of Naimark's problem has an affirmative answer for Leavitt path algebras.  One may wonder whether the same is true for any associative $k$-algebra whose dimension is countable.  (The countable dimension hypothesis is the algebraic analogue of the separability hypothesis for $C^*$-algebras.)  Specifically, one may ask

\noindent \emph{``If $A$ is an associative $k$-algebra of countable dimension, and $A$ has a unique irreducible representation up to algebraic equivalence, then is $A$ necessarily isomorphic to $M_\Lambda(k)$ for some indexing set $\Lambda$?"}

Interestingly, the answer to this question is ``No", and we describe a counterexample in Example~\ref{Cozzens-Kolchin-ex}.  This means, in particular, Theorem~\ref{Naimark-thm} is identifying a very special property of Leavitt path algebras among more general associative $k$-algebras, even when $L_k(E)$ has countable dimension or $E$ is a countable graph.
\end{remark}

\begin{example}[An Example of Cozzens and Kolchin] \label{Cozzens-Kolchin-ex}
Recall that a derivation $D$ on a ring $R$ is a function $D : R \to R$ satisfying $D(r_1 + r_2) = Dr_1 + Dr_2$ and $D(r_1r_2) = (Dr_1)r_2 + r_1D(r_2)$ for all $r_1, r_2 \in R$.    If $k$ is a field with derivation $D$, we let $k[y,D]$ denote the ring of differential polynomials in the indeterminate $y$ with coefficients in $k$; i.e., the additive group of $k[y,D]$ is the additive group of the ring of polynomials in the indeterminate $y$ with coefficients in $k$, and the multiplication in $k[y,D]$ is determined by the relation $ya := ay + D(a)$ for all $a \in k$.

If $k$ is a field of characteristic $0$ and $D$ is a derivation of $k$, Kolchin has shown in \cite{Kol} that there exists a field extension $U_{(k,D)} \supseteq k$ and a derivation $\overline{D}$ of $U_{(k,D)}$ extending $D$ with the property that for any $n \in \N$ and any polynomial $p(X) \in U_{(k,D)} [X_1, \ldots, X_n, X_{n+1}] \setminus U_{(k,D)} $ the equation
$$p(x, \overline{D}(x), \ldots, \overline{D}^{(n)} (x) ) = 0$$
has a solution $\xi \in U_{(k,D)}$.  Furthermore, every homogeneous linear differential equation in $\overline{D}$ over $U_{(k,D)}$ has a nontrivial solution in $U_{(k,D)}$.  Such a field $U_{(k,D)}$ is called a \emph{universal extension of $k$ with respect to $D$} or a \emph{universal differential field}.

Suppose $k$ is a universal differential field with derivation $D$, and define $A := k[y,D]$.  Then $A$ is a $k$-algebra with a countable basis $\{ 1, y, y^2,, \ldots \}$.  Cozzens has shown in \cite[Theorem~1.4(e)]{Coz} that $A$ has, up to isomorphism, a unique simple right $A$-module; i.e., any two irreducible representations of $A$ are algebraically equivalent.  

Cozzens has also shown in \cite[Theorem~1.4(a)(d)]{Coz} that $A$ has no zero divisors and $A$ is not a division ring.  As such, $A$ is not isomorphic to $M_\Lambda (k)$ for any index set $\Lambda$.  Consequently, $A$ is an example of a $k$-algebra with countable infinite dimension and a unique irreducible representation up to algebraic equivalence that is not isomorphic to $M_\Lambda (k)$ for any index set $\Lambda$.  
\end{example}

\section{Graph $C^*$-algebras with a Finite or Countably Infinite Number of Unitary Equivalence Classes of Irreducible Representations} \label{countably-many-repns-sec}

In this section we extend our investigations on Naimark's problem to consider graph $C^*$-algebras with a finite or countably infinite number of equivalence classes of irreducible representations.

\begin{remark} \label{extension-restriction-rem}
For a $C^*$-algebra $A$, we let $\widehat{A}$ denote the collection of unitary equivalence classes of irreducible representations, and we follow the convention of writing $\pi \in A$ and letting $\pi$ denote both the unitary equivalence class and the choice of an irreducible representation that is a representative of that equivalence class.  Recalling that two unitarily equivalent representations have equal kernels, we observe that for any subset $S \subseteq A$ it is well-defined to discuss whether $\pi (S) = 0$ or $\pi (S) \neq 0$.   With this in mind, for an ideal $I$ of $A$ we define
$$\widehat{A}^I := \{ \pi \in \widehat{A} : \pi (I) \neq 0 \} \qquad \text{ and } \qquad \widehat{A}_I := \{ \pi \in \widehat{A} : \pi (I) = 0 \}.$$

If $I$ is an ideal of $A$, we let $q_I : A \to A/I$ denote the quotient map determined by $I$.  If $\pi$ is a representation of $A$ with $\pi (I) \neq 0$, we let $\pi|_I$ denote the restriction of $\pi$ to $I$.  If $\pi$ is a representation of $A$ with $\pi (I) = 0$, we let $\pi'$ denote the representation of $A/I$ having the property that $\pi' \circ q_I = \pi$, which is provided by the first isomorphism theorem.

It is a standard fact, and a statement and proof using our notation can be found in \cite[Ch.2, \S11, Proposition~2.11.2] {Dix}, that whenever $A$ is a $C^*$-algebra and $I$ is an ideal of $A$,
\begin{itemize}
\item[(a)] there is a bijection from $\widehat{A}^I$ onto $\widehat{I}$ given by $\pi \mapsto \pi|_I$, and
\item[(b)] there is a bijection from $\widehat{A}_I$ onto $\widehat{ (A/I)}$ given by $\pi \mapsto \pi'$, where $\pi'$ is the unique representation of $A/I$ with $\pi' \circ q_I = \pi$.
\end{itemize}

Implicit in results (a) and (b) are a number of facts:  The result in (a) contains the fact that if $\pi$ is an irreducible representation of $A$ with $\pi(I) \neq 0$, then the restriction $\pi |_I$ is an irreducible representation of $I$, the surjectivity in (a) implies that any (nonzero) irreducible representation on $I$ can be extended to an irreducible representation of $A$, and the injectivity in (a) implies this extension is unique up to unitary equivalence. 

The result in (b) contains the fact that if $\pi$ is an irreducible representation of $A$ that vanishes on $I$, then the induced representation $\pi'$ on $A / I$ (for which $\pi' \circ q_I = \pi$) is an irreducible representation of $A/I$, the surjectivity in (b) implies that for any irreducible representation $\pi$ of $A/I$, the composition $\pi \circ q_I$ is an irreducible representation of $A$, and the injectivity in (b) implies that if $\pi_1$ and $\pi_2$ are irreducible representations of $A/I$, then unitary equivalence of $\pi_1 \circ q_I$ and $\pi_2 \circ q_I$ implies unitary equivalence of $\pi_1$ and $\pi_2$.

Finally, we note that since $\widehat{A} = \widehat{A}^I \sqcup \widehat{A}_I$, we have $|\widehat{A}| = |\widehat{A}^I | + | \widehat{A}_I |$, and applying (a) and (b) above we have
\begin{equation} \label{hat-A-equals-sum-eq}
|\widehat{A}| = |\widehat{I}| + |\widehat{A/I}| .
\end{equation}

It is shown in \cite[Theorem~5.1]{RT} that for any admissible pair $(H,S)$ the ideal $\mathcal{I}_{(H,S)}$ is isomorphic to a graph $C^*$-algebra.  In particular, we shall need the explicit description of $\mathcal{I}_{(H, \emptyset)}$ as a graph $C^*$-algebra.

For a saturated hereditary subset $H$ of a graph $E$, we define
$$F (H,\emptyset):= \{ \alpha \in E^* : \alpha = e_1 \ldots e_n \text{ with } r(e_n) \in H \text{ and } s(e_n) \notin H \}.$$
We also let $\overline{F}_1 (H,\emptyset)$ denote another copy of $F_1 (H,\emptyset)$ and write $\overline{\alpha}$ for the copy of $\alpha$ in $F_1 (H,\emptyset)$.  We then define $\overline{E}_{(H,\emptyset)}$ to be the graph with
\begin{align*}
\overline{E}_{(H,\emptyset)}^0 &:= H \cup F_1 (H,\emptyset) \\
\overline{E}_{(H,\emptyset)}^1 &:= \{ e \in E^1 : s(e) \in H \} \cup \overline{F}_1 (H,\emptyset)
\end{align*}
and we extend $r$ and $s$ to $\overline{E}_{(H,\emptyset)}^1$ by defining $s(\overline{\alpha}) := \alpha$ and $r(\overline{\alpha}) := r(\alpha)$.  It is shown in \cite[Theorem~5.1]{RT} that $\mathcal{I}_{(H,\emptyset)}$ is isomorphic to $C^*( \overline{E}_{(H,\emptyset)} )$.  (In fact, \cite[Theorem~5.1]{RT} gives a description of $\overline{E}_{(H,S)}$ for any admissible pair $(H,S)$ and shows $\mathcal{I}_{(H,S)}$ is isomorphic to $C^*( \overline{E}_{(H,S)} )$, but we shall only need the description for $S = \emptyset$.)
\end{remark}

The following may be considered a graph-theoretic analogue of \eqref{hat-A-equals-sum-eq} from Remark~\ref{extension-restriction-rem}.  Recall the definition of the graph $E \setminus (H, \emptyset)$ can be found in Remark~\ref{ideals-quotients-rem}.

\begin{lemma} \label{st-classes-ideal-quotient-lem}
Let $E = (E^0, E^1, r, s)$ be a graph and let $H \subseteq E^0$ be a saturated hereditary subset of vertices.  Then $| \widetilde{\partial} E | = | \widetilde{\partial} \, \overline{E}_{(H,\emptyset)} | + | \widetilde{\partial} \, (E \setminus (H, \emptyset)) |$.
\end{lemma}

\begin{proof}
To begin note that when forming $E \setminus (H, \emptyset)$, any breaking vertex in $B_H$ (which is an infinite emitter in $E$) is replaced by (and hence may be identified) with a sink in $E \setminus (H, \emptyset)$.  Furthermore, any boundary path in $E$ ending at a breaking vertex in $B_H$ corresponds bijectively to a boundary path in $E \setminus (H, \emptyset)$ ending at the corresponding sink.

Since $H$ is hereditary, we see that any boundary path in $E$ either eventually enters $H$ and stays in $H$, or is entirely outside $H$.  If a boundary path is entirely outside of $H$, then (using the identification of breaking vertices in $E$ with sinks in $E \setminus (H, \emptyset)$ described in the precious paragraph) we have a bijective correspondence between shift-tail equivalence classes of boundary paths of $E$ outside of $H$ with shift-tail equivalence classes of boundary paths in $E \setminus (H, \emptyset)$.

On the other hand, if a boundary path in $E$ enters $H$ and stays in $H$, then by truncating the initial portion of the boundary path we obtain a boundary path in $\overline{E}_{(H,\emptyset)}$.  We also see that this assignment preserves shift-tail equivalence; that is, two such boundary paths in $E$ are shift-tail equivalent in $E$ if and only if the truncated boundary paths in $\overline{E}_{(H,\emptyset)}$ are shift-tail equivalent in $\overline{E}_{(H,\emptyset)}$.  In addition, every boundary path in $\overline{E}_{(H,\emptyset)}$ is shift tail equivalent to the path entirely in $H$ by (if necessary) truncating the path's initial edge outside $H$.  Thus this assignment is a bijection, and we have a bijective correspondence between shift-tail equivalence classes of boundary paths in $E$ that enter and stay in $H$ with the shift-tail equivalence classes of boundary paths in $\overline{E}_{(H,\emptyset)}$.

Hence the shift-tail equivalence classes of boundary paths in $E$ that enter $H$ (and thus stay in $H$) are in bijective correspondence with shift-tail equivalence classes of boundary paths of $\overline{E}_{(H,\emptyset)}$, and the shift-tail equivalence classes of boundary paths in $E$ that stay outside of $H$ are in bijective correspondence with shift-tail equivalence classes of boundary paths of $E \setminus (H, \emptyset)$.  Consequently, we may conclude $| \widetilde{\partial} E | = | \widetilde{\partial} \, \overline{E}_{(H,\emptyset)} | + | \widetilde{\partial} \, (E \setminus (H, \emptyset)) |$. 
\end{proof}

We shall consider composition series for $C^*$-algebras indexed by a segment of ordinal numbers.  In the following definition recall that a \emph{limit ordinal} is an ordinal number that does not have an immediate predecessor.  (Also, note that we adopt the convention that the ordinal number $0$, not having an immediate predecessor, is a limit ordinal.)
 
\begin{definition}
A \emph{composition series} for a $C^*$-algebra $A$ consists of a strictly increasing family of closed two-sided ideals $\{ I_\gamma \}_{0 \leq \gamma \leq \delta}$ indexed by a segment $\{ 0 \leq \gamma \leq \delta \}$ for some ordinal number $\delta$ that satisfies
\begin{itemize}
\item[(i)] $I_0 = \{ 0 \}$ and $I_\delta = A$,
\item[(ii)] $I_{\gamma_1} \subsetneq I_{\gamma_2}$ whenever $0 \leq \gamma_1 < \gamma_2 \leq \delta$, and
\item[(iii)] for each limit ordinal $\epsilon \leq \delta$ we have
$$I_\epsilon = \overline{ \bigcup_{\gamma < \epsilon} I_\gamma}.$$
\end{itemize}
We often write 
$$\{ 0 \} = I_0 \subsetneq I_1 \subsetneq \ldots \subsetneq I_\delta = A$$
to denote a composition series.  Given such a composition series, we say that the composition series has \emph{length} equal to the cardinality of $\delta$.  Also, when $\gamma$ is an ordinal that is not a limit ordinal, we call the quotient $I_\gamma / I_{\gamma-1}$ the $\gamma$\textsuperscript{th} factor of the composition series.  Note that a composition series of length $\delta$ has one factor for each non-limit ordinal less than or equal to $\delta$, and hence the cardinality of the set of factors is equal to the cardinality of $\delta$.

Recall that a $C^*$-algebra is called an \emph{elementary $C^*$-algebra} if it is isomorphic to $\K(\Hi)$ for some Hilbert space $\Hi$.  With this terminology in mind, we say that a composition series is an \emph{elementary composition series} if each factor in the composition series is an elementary $C^*$-algebra; that is, if for each ordinal $\gamma \leq \delta$ that is not a limit ordinal there exists a Hilbert space $\Hi_\gamma$ such that $I_{\gamma} / I_{\gamma-1} \cong \K (\Hi_\gamma)$.
\end{definition}

\begin{lemma} \label{elementary-comp-series-gives-spectrum-lem}
Suppose $A$ is a $C^*$-algebra with an elementary composition series $\{ I_\gamma \}_{0 \leq \gamma \leq \delta}$.  Then for each ordinal $\gamma \leq \delta$ that is not a limit ordinal, there exists an irreducible representation $\pi_\gamma$ of $A$ satisfying the following:
\begin{itemize}
\item[(a)] If $\gamma_1$ and $\gamma_2$ are distinct non-limit ordinals less than or equal to $\delta$, then $\pi_{\gamma_1}$ and $\pi_{\gamma_2}$ are not unitarily equivalent.  
\item[(b)] If $\pi$ is any irreducible representation of $A$ and $\gamma$ is the smallest ordinal such that $\pi (I_\gamma) \neq 0$, then $\gamma$ is not a limit ordinal and $\pi$ is unitarily equivalent to $\pi_\gamma$.
\end{itemize}
Consequently, 
$$ \widehat{A} = \{ \pi_\gamma : \text{$0 \leq \gamma \leq \delta$ and $\gamma$ is not a limit ordinal} \}$$
and the cardinality of $\widehat{A}$ is equal to the length of the composition series (i.e., the cardinality of $\delta$).
\end{lemma}

\begin{proof}
Because the composition series is elementary, for each ordinal $\gamma \leq \delta$ that is not a limit ordinal, we have $I_{\gamma} / I_{\gamma-1} \cong \K (\Hi_\gamma)$ for some Hilbert space $\Hi_\gamma$.  For each such $\gamma$ let $\psi_\gamma : I_{\gamma} / I_{\gamma-1} \to \B (\Hi_\gamma)$ be an irreducible representation and let $q_\gamma : I_\gamma \to I_{\gamma} / I_{\gamma-1}$ be the quotient map.  Then $\psi_\gamma \circ q_{\gamma-1} :  I_{\gamma} \to \B (\Hi_{\gamma})$ is an irreducible representation, and hence (see Remark~\ref{extension-restriction-rem}(a)) there exists an irreducible representation $\pi_\gamma : A \to \B (\Hi_\gamma)$ with the property that $\pi_\gamma|_{I_\gamma} = \psi_\gamma \circ q_{\gamma -1}$.  This establishes the existence of the representation $\pi_\gamma$ for each $\gamma \leq \delta$ such that $\gamma$ is not a limit ordinal.

In addition, note that $\ker \pi_\gamma = I_{\gamma - 1}$.  Thus distinct $\pi_\gamma$'s have distinct kernels, and the $\pi_\gamma$'s are pairwise not unitarily equivalent.  This establishes the claim in (a).

Furthermore, let $\pi : A \to \B (\Hi)$ be any irreducible representation of $A$.  Since $\pi$ is nonzero, we have $\pi(I_\delta) = \pi(A) \neq 0$, and we may define $\gamma$ to be the least ordinal less than or equal to $\delta$ such that $\pi( I_\gamma) \neq 0$.  Observe that $\gamma$ cannot be a limit ordinal, because if it were we would have $\pi(I_\eta) = 0$ for all $\eta < \gamma$, implying $\pi (I_\gamma) = \pi (\overline{\bigcup_{\eta < \gamma} I_\eta }) = 0$, which is a contradiction.  Since $\pi (I_\gamma) \neq 0$ and $I_\gamma$ is an ideal of $A$, the restriction $\pi |_{I_\gamma} : I_\gamma \to \B (\Hi)$ is an irreducible representation of $I_\gamma$ (see Remark~\ref{extension-restriction-rem}(a)).  In addition, since $\gamma$ is the least ordinal with $\pi (I_\gamma) \neq 0$ and $\gamma$ is not a limit ordinal, we conclude $\pi (I_{\gamma-1}) = 0$.   Thus $\pi$ induces an irreducible representation $\widetilde {\pi} : I_\gamma / I_{\gamma - 1} \to \B (\Hi)$ by $\widetilde{\pi} (a + I_{\gamma -1} ) := \pi(a)$.  In particular, $\pi |_{I_\gamma} = \widetilde{\pi} \circ q_\gamma$.  Since any two irreducible representations of $I_\gamma / I_{\gamma - 1} \cong \K(\Hi_\gamma)$ are unitarily equivalent, we conclude $\widetilde{\pi}$ is unitarily equivalent to $\psi_\gamma$.  It follows that $\pi  |_{I_\gamma} = \widetilde{\pi} \circ q_{\gamma-1}$ is unitarily equivalent to $\psi_\gamma \circ q_{\gamma -1}$.  Furthermore, $\pi_\gamma : A \to \B (\Hi_\gamma)$ is an irreducible representation constructed so that $\pi_\gamma |_{I_\gamma} = \psi_\gamma \circ q_{\gamma -1}$.  Hence $\pi  |_{I_\gamma}$ is unitarily equivalent to $\pi_\gamma |_{I_\gamma}$, and it follows (see Remark~\ref{extension-restriction-rem}(b)) that $\pi$ is unitarily equivalent to $\pi_\gamma$.  This establishes the claim in (b).

Since the $\pi_\gamma$'s are pairwise not unitarily equivalent and every irreducible representation of $A$ is unitarily equivalent to some $\pi_\gamma$, it follows that 
$$ \widehat{A} = \{ \pi_\gamma : \text{$0 \leq \gamma \leq \delta$ and $\gamma$ is not a limit ordinal} \}$$
and the cardinality of $\widehat{A}$ is equal to the length of the composition series.
\end{proof}

Recall that in the following theorem, the term ``countable" means either finite or countably infinite.

\begin{theorem} \label{countable-infinite-irrep-st-equiv-comp-thm}
If $E = (E^0, E^1, r, s)$ is a graph, then the following are equivalent.
\begin{itemize}
\item[(1)] The collection of unitary equivalence classes of irreducible representations of $C^*(E)$ is countable.
\item[(2)] $E$ has no cycles and the collection of shift-tail equivalence classes of boundary paths of $E$ is countable. 
\item[(3)] $C^*(E)$ has an elementary composition series of countable length.
\end{itemize}
Moreover, when one (and hence all) of the above hold, $C^*(E)$ is a Type~I AF-algebra, $| \widehat{C^*(E)} | = |  \widetilde{\partial} E |$, and
$$\widehat{C^*(E)} = \{ \pi_{ [\alpha] } : [ \alpha ] \in \widetilde{\partial} E \},$$ where $\pi_{[\alpha]}$ is the representation defined in Proposition~\ref{representation-exists-prop}, Lemma~\ref{inv-subspace-with-basis-elt-all-lem}, and Notation~\ref{rep-for-st-class-not}.
\end{theorem}

\begin{proof}
(1) $\implies$ (2): Since $C^*(E)$ has a countable number of unitary equivalence classes of irreducible representations, it follows from Lemma~\ref{C-star-countable-no-cycles-lem} that $E$ has no cycles.  In addition, Corollary~\ref{number-STE-algebraic-unitary-equiv-classes-cor}(1) implies that the cardinality of the shift-tail equivalence classes of boundary paths of $E$ is countable. 

(2) $\implies$ (3): Suppose $E$ is a graph with no cycles and a countably infinite number of shift-tail equivalence classes.   We shall produce an elementary composition series via transfinite induction.  For the base case, Lemma~\ref{ste-implies-line-point-lem} implies $E$ contains a line point $w_1 \in E^0$.  Let $H_1 := T(w_1)$ and define $I_1 := \mathcal{I}_{H_1} = \mathcal{I}_{\overline{H_1} } = \mathcal{I}_{(\overline{H_1},\emptyset)}$.  By Lemma~\ref{ideal-of-line-point-matrix-and-compacts-lem}, $I_1 := \mathcal{I}_{( \overline{H_1},\emptyset)} \cong \K ( \Hi_1)$ for some Hilbert space $\Hi_1$.

For the inductive step, suppose $\gamma$ is an ordinal number for which an ideal $I_\gamma$ of $C^*(E)$ has been chosen.  Since $E$ has no cycles, all ideals in $C^*(E)$ are gauge-invariant and $I_\gamma = I_{(H_\gamma, S_\gamma)}$ for some saturated hereditary subset $H_\gamma \subseteq E^0$ and some set $S_\gamma \subseteq B_{H_\gamma}$ of breaking vertices.  Thus $C^*(E) / I_\gamma = C^*(E) /  \mathcal{I}_{(H_\gamma, S_\gamma)} \cong C^*( E \setminus (H_\gamma, S_\gamma))$.  The graph $E \setminus (H_\gamma, S_\gamma))$ consists of the subgraph $E \setminus H_\gamma$ of $E$ together with a source vertex and edges from that source vertex into the subgraph $E \setminus H_\gamma$ for each element of $B_{H_\gamma} \setminus S_\gamma$.  Consequently, each shift-tail equivalence class of $E \setminus (H_\gamma, S_\gamma)$ corresponds in a one-to-one manner to a shift-tail equivalence class of $E \setminus H_\gamma$.  Since $E \setminus H_\gamma$ is a subgraph of $E$, and since $E$ has a countable number of shift-tail equivalence classes, we may conclude that the subgraph $E \setminus H_\gamma$, and hence also the graph $E \setminus (H_\gamma, S_\gamma)$, has a countable number of shift-tail equivalence classes.

Since the graph $E \setminus (H_\gamma, S_\gamma)$ has a countable number of shift-tail equivalence classes, it follows from Lemma~\ref{ste-implies-line-point-lem} that $E \setminus (H_\gamma, S_\gamma)$ contains a line point $w_{\gamma+1}$, and moreover we can choose $w_{\gamma+1} \in E \setminus H_\gamma$.  Let $K := T(w_{\gamma+1})$ be the hereditary subset in $E \setminus (H_\gamma, S_\gamma)$ generated by $w_{\gamma+1}$, and let $J_{\gamma +1} := \mathcal{I}_{K} = \mathcal{I}_{\overline{K}} = \mathcal{I}_{(\overline{K}, \emptyset)}$ be the ideal in $C^*( E \setminus (H_\gamma, S_\gamma))$ corresponding to the admissible pair $(\overline{K}, \emptyset)$.  By Lemma~\ref{ideal-of-line-point-matrix-and-compacts-lem} $J_{\gamma +1} \cong \K(\Hi_{\gamma+1})$ for some Hilbert space $\Hi_{\gamma+1}$.  

Since $C^*( E \setminus (H_\gamma, S_\gamma)) \cong C^*(E) / I_\gamma$, with this identification we have  $J_{\gamma +1} = I_{\gamma+1} / I_\gamma$ for some ideal $I_{\gamma+1}$ of $C^*(E)$ with $I_\gamma \subsetneq I_{\gamma+1}$.  We define $I_{\gamma+1}$ as our ideal corresponding to the ordinal $\gamma+1$ and note that 
$$  I_{\gamma+1} / I_\gamma = J_{\gamma +1}  \cong \K(\Hi_{\gamma+1})$$
for the Hilbert space $\Hi_{\gamma+1}$.  

Finally, if $\epsilon$ is a limit ordinal and $I_\gamma$ has been defined for every ordinal $\gamma < \epsilon$, we define $I_\epsilon := \overline{ \bigcup_{\gamma < \epsilon} I_\gamma }$.

By transfinite induction, this construction produces an elementary composition series $\{ I_\gamma \}_{0 \leq \gamma \leq \delta}$ for $C^*(E)$.  All that remains is to show that this composition series is countably infinite.  If $\gamma$ is an ordinal less than or equal to $\delta$ and $\gamma$ is not a limit ordinal, observe that our construction produces a line point $w_\gamma \in E^0 \setminus H_\gamma$.  Moreover, for distinct $\gamma$'s the line points $w_\gamma$'s correspond to distinct shift-tail equivalence classes in $E$.  Hence, by our hypothesis in (2) that there $E$ has only countably many shift-tail equivalence classes, we conclude there are countably many non-limit ordinals less than $\delta$.  Consequently, $\delta$ is countable, and we have a countable elementary composition series.  

(3) $\implies$ (1):  If $C^*(E)$ has a countable elementary composition series, then Lemma~\ref{elementary-comp-series-gives-spectrum-lem} implies that $\widehat{C^*(E)}$ is countable.  Hence $C^*(E)$ has a countable number of unitary equivalence classes of irreducible representations.
 
This establishes the equivalence of (1) $\iff$ (2) $\iff$ (3).  It remains to show the final statements of the theorem hold in the presence of these conditions.
 
Consider the situation when one (and hence all) of Conditions (1)--(3) hold.  Since $E$ has no cycles by (2), $C^*(E)$ is an AF-algebra.  In addition, since the collection of unitary equivalence classes of irreducible representations $C^*(E)$ is countable from (1), we conclude $C^*(E)$ is a Type~I $C^*$-algebra.  Furthermore, suppose $\{ I_\gamma \}_{0 \leq \gamma \leq \delta}$ is the elementary composition series constructed in the proof of $(2) \implies (3)$.  By Lemma~\ref{elementary-comp-series-gives-spectrum-lem}
\begin{equation} \label{pi-gamma-collection-eq}
\widehat{C^*(E)} = \{ \pi_\gamma : \text{$0 \leq \gamma \leq \delta$ and $\gamma$ is not a limit ordinal} \}
\end{equation}
and if $\pi$ is any irreducible representation of $C^*(E)$ and $\gamma$ is the smallest ordinal such that $\pi (I_\gamma) \neq 0$, then $\gamma$ is not a limit ordinal and $\pi$ is unitarily equivalent to $\pi_\gamma$.

Suppose $\gamma \leq \delta$ is an ordinal that is not a limit ordinal.  As in the proof of $(2) \implies (3)$, we may write $I_{\gamma-1} = I_{ (H_{\gamma-1}, S_{\gamma-1})}$ for some admissible pair $(H_{\gamma-1}, S_{\gamma-1})$.  Also, as in the proof of $(2) \implies (3)$, there is a line point $w_{\gamma}$ in the graph $E \setminus (H_{\gamma-1}, S_{\gamma-1})$, and $I_\gamma$ is defined so $p_{w_\gamma} \in I_\gamma$.  Let $\alpha$ be the unique boundary path in $E \setminus (H_{\gamma-1}, S_{\gamma-1})$ beginning at $w_{\gamma}$, and note that $\alpha$ is also a boundary path in $E$.  Furthermore, the representation $\pi_{ [\alpha] } : C^*(E) \to \B (\Hi)$ has the property that $\pi (I_{\gamma - 1} ) = 0$ (since no vertex in $(H_{\gamma -1}, S_{\gamma-1})$ can reach $\alpha$), and $\pi (I_{\gamma } ) \neq 0$ (since $\pi_{ [\alpha]} (p_v) \alpha = \alpha$).  Consequently, $\gamma$ is the smallest smallest ordinal such that $\pi_{ [ \alpha ]}(I_\gamma) \neq 0$, and hence $\pi_{ [ \alpha ]}$ is unitarily equivalent to $\pi_\gamma$.  We thus conclude every representation in the collection of \eqref{pi-gamma-collection-eq} is unitarily equivalent to $\pi_{ [ \alpha ] }$ for some boundary path $\alpha$.  Hence
$$\widehat{C^*(E)} = \{ \pi_\gamma : \text{$0 \leq \gamma \leq \delta$ and $\gamma$ is not a limit ordinal} \} = \{ \pi_{ [\alpha] } : [ \alpha ] \in \widetilde{\partial} E \}.$$
From this it also follows that $| \widehat{C^*(E)} | = |  \widetilde{\partial} E |$.
\end{proof}

\begin{remark} \label{separable-implies-countable-comp-rem}
If $A$ is a separable $C^*$-algebra, then any composition series of $A$ is countable.  The proof of this is fairly easy:  If $\{ I_\gamma \}_{ 0 \leq \gamma \leq \delta}$ is a composition series, then for each $\gamma \leq \delta$, we may choose an element $a_\gamma \in I_\gamma$ such that $\| a_\gamma - a \| \geq 1$ for all $a \in I_{\gamma-1}$.  The collection $\{ a_\gamma : 0 \leq \gamma \leq \delta \}$ then has the property that $\| a_{\gamma_1} - a_{\gamma_1} \| \geq 1$.  Thus if $A$ is separable, the collection of $a_\gamma$'s must be countable, and $\delta$ must be countable.

As a consequence, if $E$ is a countable graph (so that $C^*(E)$ is a separable $C^*$-algebra), then Condition~(3) of Theorem~\ref{countable-infinite-irrep-st-equiv-comp-thm} may be replaced by the condition ``$C^*(E)$ has an elementary composition series" since any composition series of $C^*(E)$ automatically has countable length.
\end{remark}

\begin{remark}
It is worth noting that Theorem~\ref{countable-infinite-irrep-st-equiv-comp-thm} does not hold when one replaces the graph $C^*$-algebra $C^*(E)$ by the Leavitt path algebra $L_k(E)$ (for any choice of field $k$).  In particular, if $E$ is the graph

\medskip

$$
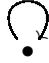
\xymatrix{ \bullet \ar@(ul,ur)[] }$$
consisting of a single vertex and a single edge, then $L_k(E)$ is the $k$-algebras of Laurent polynomials $k[x^{-1}, x]$, which has a countable number of algebraic equivalence classes of irreducible representations.  However, $E$ has exactly one boundary path, and hence exactly one shift-tail equivalence class of boundary paths.  Hence the algebraic equivalence classes of irreducible representations of $L_k(E)$ are not in one-to-one correspondence with the shift-tail equivalence classes of boundary paths of $E$, and moreover, there are representations of $L_k(E)$ that are not algebraically equivalent to any representation $\rho_{ [\alpha ]}$ coming from a boundary path $\alpha$.  

There is a characterization of when a Leavitt path algebra has a countable number of algebraic equivalence classes of representations.  Specifically,  \cite[Theorem~3.4]{AR} shows that a Leavitt path algebra over a field $k$ has a countable number of algebraic equivalence classes of representations if and only if the Leavitt path algebra has a countable composition series whose factors are countable direct sums of matrix rings over $k$ or $k[x,x^{-1}]$.  However, \cite[Theorem~3.4]{AR} makes no assertions about boundary paths or shift-tail equivalence classes of boundary paths.
\end{remark}

In the following example, we apply Theorem~\ref{countable-infinite-irrep-st-equiv-comp-thm} to compute the spectrum of a graph $C^*$-algebra.

\begin{example}
Let $E$ be the following graph.
$$
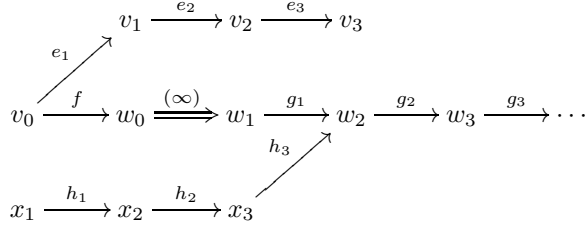
\xymatrix{
 & v_1 \ar[r]^{e_2} & v_2 \ar[r]^{e_3} & v_3  & & \\
v_0 \ar[ru]^{e_1} \ar[r]^f & w_0 \ar@{=>}[r]^{(\infty) \ } & w_1 \ar[r]^{g_1} & w_2  \ar[r]^{g_2} & w_3  \ar[r]^{g_3}  & \cdots \\
x_1 \ar[r]^{h_1} & x_2 \ar[r]^{h_2} & x_3 \ar[ru]^{h_3} & & & &
}
$$
Then $E$ has three shift-tail equivalence classes of boundary paths, and the following three boundary paths are representative chosen from each of the three classes:
\begin{align*}
\alpha_1 &:= e_1 e_2 e_3 \\
\alpha_2 &:= f \\
\alpha_3 &:= h_1 h_2 h_3 g_2 g_3 g_4 g_5 \ldots
\end{align*}
Theorem~\ref{countable-infinite-irrep-st-equiv-comp-thm} implies $\widehat{C^*(E)}$ also has three elements, and moreover $$\widehat{C^*(E)} = \{ \pi_{ [\alpha_1 ] },  \pi_{ [\alpha_2 ] },  \pi_{ [\alpha_3 ] } \}.$$

In addition, we can find an elementary composition series for $C^*(E)$ as follows:  We see that $w_2$ is a line point, and $H_1 := \{x_1, x_2, x_3, w_1, w_2, \ldots \}$ is the saturated hereditary subset generated by $w_2$.  In addition, $w_0$ is a line point (and sink) in $E \setminus (H_1, \emptyset)$,  and $H_2 := \{w_0, x_1, x_2, x_3, w_1, w_2, \ldots \}$ is the saturated hereditary subset of $E$ generated by $\{ w_0 \} \cup H_1$.  Finally $v_1$ is a line point in $E \setminus (H_2, \emptyset)$, and $E^0$ is the saturated hereditary subset of $E$ generated by $\{ v_1 \} \cup H_2$.  Hence
$$ 0 \subsetneq I_{(H_1, \emptyset)} \subsetneq I_{(H_2, \emptyset)}  \subsetneq C^*(E)$$
is an elementary composition series for $C^*(E)$ of length three.
\end{example}

Theorem~\ref{countable-infinite-irrep-st-equiv-comp-thm} shows that when $E$ is a graph with no cycles the cardinality of the shift-tail equivalence classes of boundary paths of $E$ and the cardinality of the unitary equivalence classes of irreducible representations of $C^*(E)$ are either both countable and equal or are both uncountable.  Moreover, when $E$ contains a cycle, Lemma~\ref{C-star-countable-no-cycles-lem} implies that $C^*(E)$ has uncountably many unitary equivalence classes of irreducible representations; however, the following examples show that when $E$ has a cycle, the cardinality of the shift-tail equivalence classes of boundary paths can be finite, countable, or uncountable.  Specifically, the following example shows that when $E$ contains a cycle, these cardinalities can differ in a variety of ways.

\begin{example} \label{no-relation-ste-and-spectrum-ex}
Consider the following three graphs.

$$
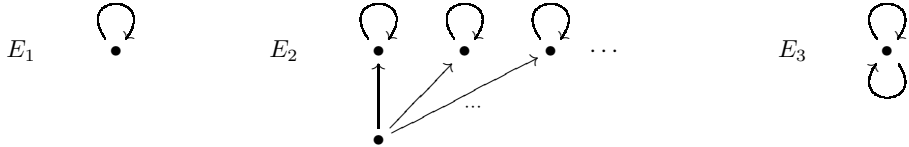
\scalebox{.93}{
\xymatrix{
E_1 & \bullet \ar@(ul,ur)[] & & E_2 & \bullet \ar@(ul,ur)[] & \bullet \ar@(ul,ur)[] & \bullet \ar@(ul,ur)[]  & \hspace{-.3in} \cdots & & E_3  & \bullet \ar@(ul,ur)[]  \ar@(dr,dl)[] & \\
&& & & \bullet \ar[u] \ar[ur] \ar[urr]_{\cdots} & & &  & & &
} }
$$
Each graph has a cycle, so by Lemma~\ref{C-star-countable-no-cycles-lem} each of $C^*(E_1)$, $C^*(E_2)$, and $C^*(E_3)$ each have an uncountable number of unitary equivalence classes of irreducible representations.  However, $E_1$ has exactly one shift-tail equivalence class of boundary paths, $E_2$ has a countably infinite number of shift-tail equivalence classes of boundary paths, and (by Lemma~\ref{two-distinct-cycles-implies-uncountable}) $E_3$ has an uncountable number of shift-tail equivalence classes of boundary paths.  Thus, in general, when $E$ contains a cycle the unitary equivalence classes of $C^*(E)$ do not correspond to shift-tail equivalence classes of boundary paths of $E$.
\end{example}

\begin{theorem}[A Trichotomy for the Spectrum of Graph $C^*$-algebras] \label{trichotomy-spectrum-thm}
$ $ \\
Let $E = (E^0, E^1, r, s)$ be a graph.  Then exactly one of the following three mutually exclusive cases holds.
\begin{itemize}
\item[] \hspace{-.4in}\textsc{Case I:} $E$ has no cycles, both $\widehat{ C^*(E)}$ and $\widetilde{\partial} E$ are countable (i.e., finite or countably infinite) with the same cardinality and $\widehat{C^*(E)} = \{ \pi_{ [\alpha] } : [ \alpha ] \in \widetilde{\partial} E \}$.

\item[] \hspace{-.4in}\textsc{Case II:}  $E$ has no cycles and both $\widehat{ C^*(E)}$ and $\widetilde{\partial} E$ are uncountably infinite.

\item[] \hspace{-.4in}\textsc{Case III:}  $E$ has at least one cycle and $\widehat{ C^*(E)}$ is uncountably infinite.  (Moreover, in this third case there are examples showing $\widetilde{\partial} E$ may be finite, countably infinite, or uncountable.)
\end{itemize}

\end{theorem}

\begin{proof}
If $E$ has no cycles, either $\widehat{ C^*(E)}$ and $\widetilde{\partial} E$ are both uncountably infinite (and we are in Case~II) or at least one of $\widehat{ C^*(E)}$ and $\widetilde{\partial} E$ is countable, and Theorem~\ref{elementary-comp-series-gives-spectrum-lem} shows that Case I holds.

If $E$ has at least one cycle, then Lemma~\ref{C-star-countable-no-cycles-lem} implies $\widehat{ C^*(E)}$ is uncountable.  Furthermore, the graphs of Example~\ref{no-relation-ste-and-spectrum-ex} show that in this case it is possible for $\widetilde{\partial} E$ to be finite, countably infinite, or uncountable.  Thus Case~III holds.
\end{proof}

\end{document}